\documentclass[12pt]{amsart}
\calclayout

\usepackage{amsthm,amssymb,latexsym}

\usepackage[abbrev]{amsrefs}
\usepackage{hyperref}

\usepackage{xcolor}
\usepackage{cite}

\newcommand{\norm}[1]{\left\|#1\right\|}                

\theoremstyle{plain}
\newtheorem{thm}{Theorem}[section]
\newtheorem{lem}[thm]{Lemma}
\newtheorem{prop}[thm]{Proposition}
\newtheorem{cor}[thm]{Corollary}
\newtheorem{ex}[thm]{Example}
\newtheorem{mainthm}{Theorem}

\theoremstyle{definition}
\newtheorem{defn}[thm]{Definition}

\theoremstyle{remark}
\newtheorem{rmk}[thm]{Remark}

\numberwithin{equation}{section}

\title[Invariant weighted Bergman metrics]{Invariant weighted Bergman metrics on domains}
\author[S. Yoo]{Sungmin Yoo}
\address{Department of Mathematics, Incheon National University, 119 Academy-ro, Yeonsu-gu, Incheon, 22012, Republic of Korea}

\begin{document}

\begin{abstract}
In this paper, we study the cases where the weighted Bergman metrics of a domain are invariant under biholomorphisms by introducing the concept of {\it invariant weight assignments}, focusing on two examples by Tian and Tsuji, respectively.
Using Bergman's minimum integral method and a domain version of the Tian-Yau-Zelditch expansion for the weighted Bergman kernels and metrics, we give an alternative proof of uniform convergence of Tian's sequence of Bergman kernels and metrics on uniform squeezing domains. 
We also present a proof of the uniform convergence of Tsuji's dynamical kernel sequence on uniform squeezing domains.
\end{abstract}

\maketitle

\section{Introduction}

Bergman \cite{bergman1970} introduced the concept of the Bergman kernel and metric for bounded domains from the Hilbert space of $L^2$-holomorphic functions, known as the Bergman space. 
This seminal work laid the foundation for understanding complex manifolds from the perspective of function theory.
Subsequently, Kobayashi \cite{kobayashi1959geometry} extended these concepts to abstract complex manifolds by considering holomorphic $(n,0)$-forms, sections of the canonical line bundle. 
The remarkable property of invariance under biholomorphisms renders the Bergman kernel form and metric canonical within the realm of Several Complex Variables and Complex Geometry.

These concepts found further generalization to polarized manifolds, which are compact complex manifolds endowed with ample line bundles \cite{tian1990,donaldson2001scalar}. 
Their significance became particularly pronounced in the proof of the celebrated Tian-Yau-Donaldson conjecture \cite{chen2015kahler1,chen2015kahler2,chen2015kahler3,tian2015k}. 
However, it is worth noting that the definition of the Bergman kernel function in this setting is contingent upon the choice of the hermitian metric and volume form in defining the inner product on the vector space of global sections. 

Returning to the domain context, these results can be elucidated through the concept of weighted Bergman kernels.
Let $\Omega$ be a domain in $\mathbb{C}^n$ and $d\lambda$ be the standard Lebesgue measure.
For a positive measurable function $\mu$ on $\Omega$, consider the space
$$
\mathcal{A}^2(\Omega,\mu):=\left\{u\in\mathcal{O}(\Omega)\ \Big|\ \norm{u}^2_{\Omega,\mu}:=\int_{\Omega}|u|^2 \mu d\lambda<\infty\right\}.
$$
If the weight function $\mu$ satisfies the admissible condition (e.g. continuous function),
then the above space admits the reproducing kernel, called the weighted Bergman kernel or $\mu$-Bergman kernel.
As in the classic case ($\mu=1_{\Omega}$), this gives us a K\"ahler metric if $\Omega$ is bounded, which is called the weighted Bergman metric or $\mu$-Bergman metric.
However, compared to the classic Bergman metric, there has been relatively little research on the geometry of weighted Bergman metrics, primarily due to their lack of biholomorphic invariance.

In this paper, we study the situation when the weighted Bergman metric is invariant under biholomorphisms,
For this, we introduce the concept of (biholomorphically) {\it invariant weight assignment}.
Roughly speaking, we consider an weight assignment $\mathcal{M}$ for domains such that the weight function $\mu_{\Omega}:=\mathcal{M}(\Omega)$ only depends on the geometry of the domain $\Omega$ (for precise statements, see Definition \ref{def: invariant assignment}).
\begin{mainthm}[Theorem \ref{thm: invariance}]
If the assignment $\mathcal{M}$ is invariant, the $\mathcal{M}$-weighted Bergman metric is invariant:
$
g_{\Omega,\mu_{\Omega}}=F^*g_{F(\Omega),\mu_{F(\Omega)}},
$
for any biholomorphism $F$.
\end{mainthm}
The above theorem yields that we can consider various types of biholomorphically {\it invariant} K\"ahler metrics on domains depending upon the choice of invariant assignments.
As examples, we investigate two important invariant weighted Bergman metrics following the approaches of Tian \cite{tian1990} and Tsuji \cite{tsuji2010} in the setting of canonically polarized compact manifolds.

As the first example, consider the sequence of weight functions
$
\mu^{\rm KE}_{\Omega,m}
=
{\det\left(g^{\rm KE}_{\Omega}\right)}^{-(m-1)},
$
where $g^{\rm KE}_{\Omega}$ is the K\"ahler-Einstein metric of the bounded pseudoconvex domain $\Omega$.
Then, the weighted Bergman kernels $K_{\Omega,\mu^{\rm KE}_{\Omega,m}}$ gives us a sequence of {\it invariant} weighted Bergman metrics $g_{\Omega,\mu^{\rm KE}_{\Omega,m}}$.
Denotes the corresponding holomorphic sectional curvature by ${H}_{\Omega,\mu^{\rm KE}_{\Omega,m}}$.
An appropriate normalization gives the following
\begin{mainthm}[Theorem \ref{thm: tian ke seq}]
If $\Omega$ is a bounded pseudoconvex domain (with the uniform squeezing property),
then $\sqrt[m]{{K}_{\Omega,\mu^{\rm KE}_{\Omega,m}}}, \frac{1}{m}{g}_{\Omega,\mu^{\rm KE}_{\Omega,m}}, m{H}_{\Omega,\mu^{\rm KE}_{\Omega,m}}$ converges (uniformly) on $\Omega$ to $\det\left(g^{\rm KE}_{\Omega}\right), g^{\rm KE}_{\Omega}, H_{g^{\rm KE}_{\Omega}}$, respectively, as $m\rightarrow\infty$.
\end{mainthm}
This theorem can be considered as a noncompact version of Tian's theorem in \cite{tian1990} for domains with the K\"ahler-Einstein metrics.
For the proof, Tian \cite{tian1990} constructed an asymptotic expansion of the Bergman metric of canonically polarized manifolds using the ``peaked section" method, based on H\"ormander's theorem.
For domains in $\mathbb{C}^n$, this method can be understood in terms of the weighted version of the Bergman minimum integral method.
In this paper, we present a straightforward proof of the following domain version of the Tian-Yau-Zelditch expansion for the weighted Bergman kernels, metrics, and curvatures, following the proofs in \cite{tian1990,ruan1998canonical,lu2000lower,szekelyhidi2014introduction}.

\begin{mainthm}[Theorem \ref{thm: asymptototic expansion}]
Let $\Omega$ be a pseudoconvex domain in $\mathbb{C}^n$ with a smooth strictly plurisubharmonic function $\varphi$.
Suppose that $(\Omega,e^{-\varphi})$ admits the positive definite weighted Bergman metric.
Fix a point $p\in\Omega$ and a vector $X\in\mathbb{C}^n$.
For sufficiently large $m$, we have 
\begin{align*}
K_{\Omega,e^{-m\varphi}}(p)
&=
\frac{\det{\left(\varphi_{k\overline{l}}(p)\right)}}{e^{-m\varphi(p)}}
\left(\frac{m}{\pi}\right)^n
\left(1-\frac{S_{\varphi}(p)}{2m}+O\Big(\frac{1}{m^{2}}\Big)\right),
\\
g_{\Omega,e^{-m\varphi}}(p;X)
&=
mg_{\varphi}(p;X)
\left(1-\frac{R_{\varphi}(p;X)}{m}+O\Big(\frac{1}{m^{2}}\Big)\right),
\\
H_{\Omega,e^{-m\varphi}}(p;X)
&=
\frac{H_{\varphi}(p;X)}{m}+O\Big(\frac{1}{m^{2}}\Big),
\end{align*}
where $S_{\varphi},R_{\varphi},H_{\varphi}$ are the scalar, Ricci, holomorphic sectional curvature of the K\"ahler metric $g_{\varphi}:=i\partial\overline{\partial}\varphi$, respectively. 
\end{mainthm}
Note that the above formula for $K_{\Omega,e^{-m\varphi}}$ differs slightly from the compact case, since the definition of the weighted Bergman kernel for domains depends on the standard Euclidean coordinates.
For the proof, we use a classic version of H\"ormander's theorem for domains \cite{hormander1965} as we use the standard Lebesgue measure $d\lambda$ (the volume form of the {\it incomplete} Euclidean metric) instead of $(i\partial\overline{\partial}\varphi)^n=\det{\left(\varphi_{k\overline{l}}\right)}d\lambda$.
Note that in our case, the metric $i\partial\overline{\partial}\varphi$ not necessarily to be complete on $\Omega$.

As the second example of {\it invariant} weighted Bergman metrics, we consider an iterative sequence of weighted functions, developed by Tsuji's \cite{tsuji2010}.
Set $\widetilde{\mu}^{\rm B}_{\Omega,1}:=1_{\Omega}$.
For the given weight function $\widetilde{\mu}^{\rm B}_{\Omega,m}$, consider the weighted Bergman kernel $K_{\Omega,\widetilde{\mu}^{\rm B}_{\Omega,m}}$ of the space $\mathcal{A}^2_{\widetilde{\mu}^{\rm B}_{\Omega,m}}(\Omega)$.
Define the next weight function by
$$
\widetilde{\mu}^{\rm B}_{\Omega,m+1}
:=
\frac{1}{C_m}\frac{1}{K_{\Omega,\widetilde{\mu}^{\rm B}_{\Omega,m}}}
:=
\left(\frac{m}{\pi}\right)^n\left(1-\frac{n}{2m}\right)
\frac{1}{K_{\Omega,\widetilde{\mu}^{\rm B}_{\Omega,m}}}.
$$
Then the corresponding weighted Bergman metrics are invariant under biholomorphisms (Proposition \ref{prop: invariance of Tsuji}).
Although we were unable to obtain the metric convergence result, we have the following convergence result on the potential level.
\begin{mainthm}[Theorem \ref{thm: tsuji seq converge}]
Let $\Omega$ be a uniform squeezing domain in $\mathbb{C}^n$.
There exists a uniform constant $C>0$ satisfying
$$
e^{-\frac{C}{m}}\det\left(g^{\rm KE}_{\Omega}\right)
\leq
\sqrt[m]{C_m
K_{\Omega,\widetilde{\mu}^{\rm B}_{\Omega,m}}}
\leq
e^{\frac{C}{m}}\det\left(g^{\rm KE}_{\Omega}\right).
$$
Hence, $\frac{1}{m}\log\left(C_m
K_{\Omega,\widetilde{\mu}^{\rm B}_{\Omega,m}}\right)$
uniformly converges to $\log\left(\det\left(g^{\rm KE}_{\Omega}\right)\right)$ at the rate $\frac{1}{m}$.
\end{mainthm}

For the above uniform estimate, we leverage the Tian-Yau-Zelditch expansion, as in the proof for the case of compact canonically polarized manifolds \cite{tsuji2010,song2010}.
A difference is that we adjust the normalizing factor in the definition of the dynamical system of the Bergman kernel to achieve better convergence speed, following Berndtsson's idea in \cite{berndtsson2009} (compare with the convergence rate `$\log m/m$' in \cite{song2010}).

The results of this paper can be extended to non-compact manifolds with bounded geometry, but we focus solely on domains here.
This is because the domain theory itself is important, especially in terms of providing numerous concrete examples.

Our paper is organized as follows: In Section 2, we provide a brief overview of known results concerning weighted Bergman kernels and metrics. In Section 3, we introduce the concept of invariant weighted Bergman metrics. Section 4 presents a proof of the Tian-Yau-Zelditch expansion for domains, and in Section 5, we establish a proof of the uniform convergence of Tian and Tsuji's weighted Bergman sequence on uniform squeezing domains.

\section{Weighted Bergman kernel and metric}
In this section, we briefly review well-known results for the classic Bergman kernel,metric, and their generalization to the weighted setting including the minimum integral method.
\subsection{Weighted Bergman space, kernel, and metric}
Let $\Omega$ be a domain in $\mathbb{C}^n$ and $d\lambda$ be the Lebesgue measure of $\mathbb{C}^n$.
Consider a measure $d\mu$ on $\Omega$, which is defined by
$$
d\mu(z):=\mu(z)d\lambda(z),
$$
where $\mu(z)$ is a Lebesgue measurable positive real-valued function on $\Omega$.
The function $\mu$ is called an {\it weight} function.
Denote by $L^2(\Omega,\mu)$ the space of all Lebesgue measurable
complex-valued $\mu$-square integrable functions on $\Omega$, i.e.,
$$
L^2(\Omega,\mu):=\left\{u:\Omega\rightarrow\mathbb{C}\ \Big|\ \norm{u}^2_{\Omega,\mu}:=\int_{\Omega}|u|^2 d\mu<\infty\right\}.
$$
Then $L^2(\Omega,\mu)$ is a separable Hilbert space with respect to the inner product:
$$
\langle u,v \rangle_{\Omega,\mu}:=\int_{\Omega}u\overline{v} d\mu=\int_{\Omega}u\overline{v}\mu d\lambda.
$$
Consider the linear subspace of $L^2(\Omega,\mu)$, which is defined by
$$
\mathcal{A}^2(\Omega,\mu):=\mathcal{O}(\Omega)\cap L^2(\Omega,\mu),
$$
where $\mathcal{O}(\Omega)$ is the set of all holomorphic functions on $\Omega$.
The space $\mathcal{A}^2(\Omega,\mu)$ is called the {\it $\mu$-Bergman space} or {\it weighted Bergman space}.
Note that if $\mathcal{A}^2(\Omega,\mu)$ is a closed subspace of $L^2(\Omega,\mu)$, then $\mathcal{A}^2(\Omega,\mu)$ is also a Hilbert space.

\begin{defn}
An weight $\mu$ on $\Omega$ is called an {\it admissible} weight function if $\mathcal{A}^2(\Omega,\mu)$ is a closed subspace of $L^2(\Omega,\mu)$ and for each fixed point $p\in\Omega$, the functional
$$
\Phi_p:u\rightarrow u(p)
$$
is a continuous linear functional on $\mathcal{A}^2(\Omega,\mu)$.
\end{defn}

The following criterion is useful to check the condition for the admissibility. 

\begin{thm}[Pasternak-Winiarski \cite{pasternak1990dependence,pasternak1992weights}]\label{thm: PW criterion}
An weight function $\mu$ on $\Omega$ is admissible if and only if for any compact subset $A\subset\Omega$, there is a constant $C_A>0$ such that for all $u\in\mathcal{A}^2(\Omega,\mu)$, it satisfies the following Cauchy type inequality:
\begin{equation*}
    \sup_{z\in A}|u(z)|\leq C_A\norm{u}_{\Omega,\mu}.
\end{equation*}
\end{thm}

\begin{rmk}
It is well-known that the characteristic function $\mu:=1_{\Omega}$ is admissible.
In this case, $\mathcal{A}^2(\Omega):=\mathcal{A}^2(\Omega,1_{\Omega})$ is called the (classic or unweighted) Bergman space.
In general, if $\mu^{-1}$ is locally integrable on $\Omega$, then $\mu$ is admissible (cf. Pasternak-Winiarski \cite{pasternak1990dependence}).
Therefore, every positive continuous function on $\Omega$ is admissible.
\end{rmk}

Note that if $\mu$ is admissible, we can guarantee the existence of the kernel function applying the Riesz representation theorem, i.e., there exists the unique function $K_{\mu}$ on $\Omega\times\Omega$ such that for all $u\in\mathcal{A}^2(\Omega,\mu)$, it satisfies the reproducing property:
$$
u(z)=\int_{\Omega}K_{\mu}(z,w)u(w)d\mu(w).
$$
In this case, the function $K_{\mu}$ is called the {\it $\mu$-Bergman function} or {\it weighted Bergman kernel function}.
Similarly to the classic Bergman kernel function, we have the following

\begin{thm}[Pasternak-Winiarski \cite{pasternak1990dependence}]
For an admissible weight $\mu$, the weighted Bergman kernel function $K_{\mu}$ satisfies the following properties.
\begin{itemize}
    \item[(1)] For any complete orthonormal basis $\{u^{\nu}\}$  and compact subset $K\subset\Omega$, the series
    $$
    \sum_{\nu} u^{\nu}(z)\overline{u^{\nu}(w)}
    $$
    converges uniformly on $K\times K$ to the weighted Bergman kernel $K_{\mu}$.
    \item[(2)] $K_{\mu}$ is conjugate symmetric:
    $
    K_{\mu}(z,w)=\overline{K_{\mu}(w,z)}.
    $
    \item[(3)] $K_{\mu}$ is real analytic.
    \item[(4)] $K_{\mu}$ is the integral kernel of the orthogonal projector $P_{\mu}:L^2(\Omega,\mu)\rightarrow\mathcal{A}^2(\Omega,\mu)$, i.e., for all $u\in L^2(\Omega,\mu)$, we have
    $$
    P_{\mu}[u](z)=\int_{\Omega}K_{\mu}(z,w)u(w)d\mu(w).
    $$
\end{itemize}
\end{thm}

\begin{rmk}
If $\mu$ is admissible, the $\mu$-Bergman space $\mathcal{A}^2(\Omega,\mu)$ inherits the separability from $L^2(\Omega,\mu)$ hence a complete orthonormal basis always exists.
\end{rmk}

From now on, we always assume that our weight function $\mu$ is admissible on $\Omega$.
Then we can define a symmetric tensor 
$$
g_{\Omega,\mu}:=
\sum^n_{j,k=1}g_{j\overline{k}}(z)dz_j\otimes d\overline{z_k}
=\sum^n_{j,k=1}\frac{\partial^2\log K_{\Omega,\mu}(z,z)}{\partial z_j \partial\overline{z_k}}dz_j\otimes d\overline{z_k}
$$
for all $z\in\Omega$ satisfying $K_{\Omega,\mu}(z,z)>0$.
We will call this the {\it $\mu$-Bergman metric} or {\it weighted Bergman (psuedo)-metric} of $\Omega$.
If $g_{\Omega,\mu}(z)$ is positive-definite for all $z\in\Omega$, we say that $(\Omega,\mu)$ {\it admits} the weighted Bergman metric.
In this case, the weighted Bergman metric is a (real-analytic) K\"{a}hler metric.
Denote the corresponding Bergman length square of 
$X=\sum_{j=1}^n X_j\frac{\partial}{\partial z_j}\big|_{z=p}\in T_p\Omega\cong\mathbb{C}^n$ at $p\in\Omega$ 
by
$$
g_{\Omega,\mu}(p;X)
:=
\sum\limits_{j,k=1}^ng_{j\overline{k}}(p)X_j\overline{X_k}.
$$
The {\it holomorphic sectional curvature} of $g_{\Omega,\mu}$ at $p$ along $X\in\mathbb{C}^n$ is 
$$
H_{\Omega,\mu}(p;X)
:=\left(\sum\limits_{i,j,k,l=1}^nR_{i\overline{j}k\overline{l}}(p)X_i\overline{X_j}X_k\overline{X_l}\right)\left(\sum\limits_{j,k=1}^ng_{j\overline{k}}(p)X_j\overline{X_k}\right)^{-2},
$$
where $R_{i\overline{j}k\overline{l}}=-\frac{\partial^2g_{i\overline{j}}}{\partial z_k\partial \overline{z_l}}+g^{\alpha\overline{\beta}}\frac{\partial g_{i\overline{\beta}}}{\partial z_k}\frac{\partial g_{\alpha\overline{j}}}{\partial \overline{z_l}}$ are coefficients of the curvature tensor of $g_{\Omega,\mu}$.

\subsection{Bergman's special basis and minimum integrals}

Now we discuss a special way to construct a complete orthonormal basis of $\mathcal{A}^2(\Omega,\mu)$, which were developed by Bergman (in the unweighted case).
Suppose that $\mu$ is an admissible weight on $\Omega$.
Let $(z_1,\ldots,z_n)$ be the standard Euclidean coordinates for $\mathbb{C}^n$.


For the multi-indices
$
\alpha=(\alpha_1,\ldots,\alpha_n)\in\mathbb{N}^n
$
with
$|\alpha|:=\sum_{i=1}^n\alpha_i$
,
we will denote the holomorphic derivatives of a function $u$ by
$$
D^{\alpha}u:=D^{\alpha}_zu=\frac{\partial^{|\alpha|}}{\partial z_1^{\alpha_1}\cdots\partial z_n^{\alpha_n}}u.
$$

\begin{defn}[Bergman's special basis]
Fix a point $p\in\Omega$.
We say that a complete orthonormal basis $\{u^{\alpha}\}$ for $\mathcal{A}^2(\Omega,\mu)$ is {\it special} at $p$ if it satisfies that
$D^{\alpha}u^{\alpha}(p)\neq0$ and $D^{\beta}u^{\alpha}(p)=0$ if $\beta<\alpha$,
where the order for multi-indices is given by the lexicographic order.
\end{defn}

For each multi-index $\alpha$, define the subsets of $\mathcal{A}^2_\mu(\Omega):=\mathcal{A}^2(\Omega,\mu)$ by
$$
\mathcal{E}^{\alpha}_\mu(\Omega):=\left\{u\in\mathcal{A}^2_\mu(\Omega) : D^{\alpha}u(p)=1,\ D^{\beta}u(p)=0 \ \ {\rm if\ } \beta<\alpha\right\}
$$
If the set $\mathcal{E}^{\alpha}_\mu(\Omega)$ is non-empty, 
one can check that there exists a $L^2$-minimal element $v^{\alpha}$ in $\mathcal{E}^{\alpha}_\mu(\Omega)$.
Then $\{u^{\alpha}\}$ with $u^{\alpha}:=v^{\alpha}/\norm{v^{\alpha}}_{\Omega,\mu}$ is a complete orthonormal basis for $\mathcal{A}^2_\mu(\Omega)$, which is special at the point $p$.

\begin{prop}
If $\Omega$ is a bounded domain with an admissible weight $\mu\in L^1(\Omega)$, then the diagonal part of weighted Bergman kernel is positive everywhere, and the weighted Bergman metric is positive-definite everywhere.
\end{prop}

\begin{proof}
For the multi-indices
$
\alpha=(\alpha_1,\ldots,\alpha_n)\in\mathbb{N}^n
$
, denote the monomials (including the constant function $1=z^{\bf0}$ for ${\bf0}:=(0,\ldots,0)$) by
$$
z^{\alpha}:=z_1^{\alpha_1}\cdots z_n^{\alpha_n}.
$$
Since $\Omega$ is bounded and $\mu\in L^1(\Omega)$, $z^{\alpha}\in\mathcal{E}^{\alpha}_\mu(\Omega)$ for all $\alpha$.
Then, the subsets $\mathcal{E}^{\alpha}_\mu(\Omega)$ are non-empty for any point $p\in\Omega$ so that there exists a special basis $\{u^{\alpha}\}$ at $p$.
The conclusion follows from the following equalities.
$$
K_{\Omega,\mu}(p)
:=K_{\Omega,\mu}(p,p)
=\left|u^{\bf0}(p)\right|^2
=\left|v^{\bf0}(p)\right|^2/\norm{v^{\bf0}}^2_{\Omega,\mu},
$$
$$
  \left.  \frac{\partial^2}{\partial z_j\partial\overline{z_k}}\log K_{\Omega,\mu}(z) \right|_p
=\left|u^{\bf0}(p)\right|^{-2}\sum_{|\alpha|=1}\frac{\partial u^{\alpha}}{\partial z_j}(p)\overline{\frac{\partial u^{\alpha}}{\partial z_k}(p)}.
$$
\end{proof}

\begin{defn}
Fix a point $p\in\Omega$ and a nonzero vector $X\in\mathbb{C}^n$.
Define subsets:
\begin{align*}
\mathcal{E}^0_\mu(\Omega)&:=\left\{u\in\mathcal{A}^2_\mu(\Omega) : u(p)=1\right\}=\mathcal{E}^{\bf0}_\mu(\Omega),\\
\mathcal{E}^1_\mu(\Omega)&:=\left\{u\in\mathcal{A}^2_\mu(\Omega) : u(p)=0,\ D_Xu(p)=1 \right\},\\
\mathcal{E}^2_\mu(\Omega)&:=\left\{u\in\mathcal{A}^2_\mu(\Omega) : u(p)=0,\ du(p)=0,\ D_XD_Xu(p)=1 \right\},
\end{align*}
where $D_X$ is the directional derivative along $X$.
\end{defn}
\begin{defn}
\emph{Minimum integrals} of the weighted Bergman kernels are defined by
$$
I^0_{\Omega,\mu}(p):=\inf_{u\in \mathcal{E}^0_\mu(\Omega)}\|u\|_{\Omega,\mu}^2,\ \ \ \ 
I^1_{\Omega,\mu}(p;X):=\inf_{u\in \mathcal{E}^1_\mu(\Omega)}\|u\|_{\Omega,\mu}^2,\ \ \ \ 
I^2_{\Omega,\mu}(p;X):=\inf_{u\in \mathcal{E}^2_\mu(\Omega)}\|u\|_{\Omega,\mu}^2.
$$
\end{defn}
For simplicity, we will use the following notations from now on.
$$
I^0_{\mu}:=I^0_{\Omega,\mu}(p),\ \ \ \ 
I^1_{\mu}:=I^1_{\Omega,\mu}(p;X),\ \ \ \ 
I^2_{\mu}:=I^2_{\Omega,\mu}(p;X).
$$
$$
D_X^0u:=u,\ \ \ \ 
D_X^1u:=D_Xu,\ \ \ \ 
D_X^2u:=D_XD_Xu.
$$
We can generalize the Bergman-Fuks formula to the weighted cases, using the same proof.
For reader's convenience, we briefly sketch the proof here.
\begin{thm} [Bergman-Fuks formula]\label{thm: bergman-fuks}
$$
 K_{\Omega,\mu}(p)=\frac{1}{I^0_{\mu}},\ \ \ \  
 g_{\Omega,\mu}(p;X)=\frac{I^0_{\mu}}{I^1_{\mu}},\ \ \ \ 
 H_{\Omega,\mu}(p;X)=2-\frac{(I^1_{\mu})^2}{I^2_{\mu}I^0_{\mu}}.
$$
\end{thm}
\begin{proof}
For $j=0,1,2$, denote the minimizer of the set $\mathcal{E}^j_{\mu}(\Omega)$ by $v^j$.
Consider an orthonormal basis of $\mathcal{A}^2_\mu(\Omega)$ including
$u^j:=v^j/\|v^j\|_{\Omega,\mu}$.
Note that $u^k(p)=0$ if $k>0$, and $du^k(p)=0$ if $k>1$.
Let $K:=K_{\Omega,\mu}$, $K_X:=D_XK_{\Omega,\mu}$, and $K_{\overline{X}}:=\overline{D_X}K_{\Omega,\mu}$.
Similarly, for the directional derivatives of any function, use sub-indices.
Then
$$
K(p)=|u^0(p)|^2,\ \ 
K_{X}(p)=u_X^0(p)\cdot\overline{u^0(p)},\ \ 
K_{X\overline{X}}(p)=|u^0_X(p)|^2+|u^1_X(p)|^2.
$$
Therefore, we have
$$
g(p;X):=
g_{\Omega,\mu}(p;X)
=\frac{K_{X\overline{X}}(p)K(p)-K_{X}(p)K_{\overline{X}}(p)}{K^2(p)}
=\frac{|u^1_X(p)|^2}{|u^0(p)|^2}.
$$
Similarly, a long and tedious computation shows that
$
H_{\Omega,\mu}(p;X)
=
2-\frac{|u^0(p)|^2|u^2_{XX}(p)|^2}{|u^1_X(p)|^4}.
$
Then, the conclusion follows from 
$|D_X^ju^j(p)|^2=1/\|v^j\|_{\Omega,\mu}^2=1/I^j_{\mu}$.
\end{proof}

\subsection{The weighted Bergman kernel and metric of the ball}\label{Forelli-Rudin}
In \cite{forelli1974projections}, Forelli and Rudin computed the Bergman kernel function of the unit ball $\mathbb{B}^n(1)$ in $\mathbb{C}^n$ with the admissible weight $(1-|z|^2)^{m}$.
Using the transformation formula (\ref{eqn: transformation}), one can obtain the weighted Bergman kernel and metric of a dilated ball.
Let $\mathbb{B}^n(r)$ be the ball in $\mathbb{C}^n$ with the radius $r$ centered at the origin.
Then, for the admissible weight $\mu(w):=\left(\frac{r^2-|w|^2}{r^2}\right)^m$,
$$
K_{\mathbb{B}^n(r),\mu}(w)=\frac{1}{c_m(r)}\left(\frac{r^2}{r^2-|w|^2}\right)^{n+m+1},
$$
where $c_m(r)$ denotes the weighted volume of the ball $\mathbb{B}^n(r)$ for integers $m\geq0$:
$$
c_m(r):=\int_{\mathbb{B}^n(r)}\Big(\frac{r^2-|w|^2}{r^2}\Big)^m d\lambda(w)=(\pi r^2)^n\frac{\Gamma(m+1)}{\Gamma(n+m+1)}=(\pi r^2)^n\frac{m!}{(n+m)!}.
$$
The corresponding weighted Bergman metric is given by
$$
g_{\mathbb{B}^n(r),\mu}=
\sum\limits_{j,k=1}^n(n+m+1)\frac{(r^2-|w|^2)\delta_{j\overline{k}}+\overline{w_j}w_k}{(r^2-|w|^2)^2}dw_j\otimes d\overline{w_k}.
$$
The holomorphic sectional curvatures of $g_{\mathbb{B}^n(r),\mu}$ are constant
$
H_{\mathbb{B}^n(r),\mu}(w)
=
-\frac{2}{n+m+1}.
$


\section{Invariance of the weighted Bergman kernel and metric}

\subsection{Transformation formula for the weighted Bergman kernels}
It is well-known that the classic Bergman kernel function satisfies the transformation formula for biholomorphisms.
In the weighted case, we can generalize it as follows.

\begin{prop}\label{prop: transformation}
Let $\Omega$ and $\Omega'$ be domains in $\mathbb{C}^n$ and
$F:\Omega\rightarrow\Omega'$ be a biholomorphism.
Let $\mu$ be an admissible weight function of $\Omega$.
\begin{enumerate}
	\item Let $h$ be a non-vanishing holomorphic function on $\Omega$.
	Then the function
	$$\mu'(F(z)):=|h(z)|^2\mu(z)$$ 
	is an admissible weight function on $\Omega'$.
	\item In the above case, we have the following transformation formula:
	\begin{equation}\label{eqn: transformation}
	K_{\Omega,\mu}(z,w)=\mathcal{J}(F(z)) h(z)\cdot K_{\Omega',\mu'} (F(z),F(w))\cdot \overline{\mathcal{J}(F(w)) h(w)},
	\end{equation}
	where $\mathcal{J}(F):=\det J_{\mathbb{C}}F$ is the determinant of complex Jacobian of $F$.
\end{enumerate}	
\end{prop}

\begin{proof}
\begin{enumerate}
	\item Let $u$ be a function in $\mathcal{A}^2(\Omega',\mu')$.
	Set $\zeta=F(z)$. Then we have
	$$
	\norm{u}^2_{\Omega',\mu'}:=\int_{\Omega'}|u(\zeta)|^2 \mu'(\zeta)d\lambda(\zeta)
	=\int_{\Omega}\left|u(F(z))h(z)\mathcal{J}(F(z))\right|^2 \mu(z)d\lambda(z).
	$$
	This implies that 
	$$
	v(z):=u(F(z))h(z)\mathcal{J}(F(z))\in\mathcal{A}^2(\Omega,\mu).
	$$
	Let $A'$ be a compact subset of $\Omega'$ so that $A:=F^{-1}(A')$ is a compact subset of $\Omega$.
	Since $\mu$ is admissible, there is a constant $C_{A}>0$ such that
    $$
    \sup_{z\in A}|v(z)|\leq C_{A}\norm{v}_{\Omega,\mu}.
    $$
    Choose $C_{A'}:=C_{A}\sup\limits_{z\in A}\left|\frac{1}{h(z)\mathcal{J}(F(z))}\right|>0$.
    Then we have
    $$
    \sup_{\zeta\in A'}|u(\zeta)|
    =\sup_{z\in A}\left|v(z)\frac{1}{h(z)\mathcal{J}(F(z))}\right|\leq C_{A'}\norm{u}_{\Omega',\mu'}.
    $$
    By Theorem \ref{thm: PW criterion}, $\mu'$ is an admissible weight function of $\Omega'$.\\
    
    \item Let $\xi=F(w)$. For any function $u\in\mathcal{A}^2(\Omega,\mu)$, we have
    \begin{align*}
    &\int_{\Omega}\mathcal{J}(F(z)) h(z) K_{\Omega',\mu'} (F(z),F(w)) \overline{\mathcal{J}(F(w)) h(w)}u(w)\mu(w)d\lambda(w)\\
    =&\int_{\Omega'}\mathcal{J}(F(z)) h(z) K_{\Omega',\mu'} (F(z),\xi) \overline{\mathcal{J}(F(w)) h(F^{-1}(\xi))}u(F^{-1}(\xi))\frac{\mu'(\xi)d\lambda(\xi)}{\left|h(F^{-1}(\xi))\mathcal{J}(F(w))\right|^2}\\
    =&\int_{\Omega'}\mathcal{J}(F(z)) h(z) K_{\Omega',\mu'} (F(z),\xi) \left(\mathcal{J}\left(F\left(w\right)\right) h\left(F^{-1}\left(\xi\right)\right)\right)^{-1}u(F^{-1}(\xi))d\mu'(\xi)\\
    =&\mathcal{J}(F(z)) h(z)\left(\mathcal{J}(F(z)) h(z)\right)^{-1}u(F^{-1}(F(z)))=u(z)
    \end{align*}
    The uniqueness of the kernel function implies the equation (\ref{eqn: transformation}).
\end{enumerate}	
\end{proof}

\begin{rmk}
The above proposition is a generalization of Lemma 1 in \cite{dragomir1994weighted} when $h\equiv1$.
\end{rmk}

\subsection{Invariant weighted Bergman metrics}

Recall that in the unweighted case ($\mu_\Omega=1_\Omega$), the definitions of the classic Bergman kernel and metric depend only on the geometry of domain $\Omega$.
One of the most important properties of the classic Bergman metric is that it is {\it invariant under biholomorphisms}, in the sense that for any biholomorphism $F$, we have
$$
g_{\Omega,1_{\Omega}}=F^*g_{F(\Omega),1_{F(\Omega)}}.
$$
Since this invariance comes from the transformation formula of Bergman kernels under biholomorphisms, we need to check the following weighted version, which is a generalization of Theorem 1 in \cite{dragomir1994weighted} when $h\equiv1$.

\begin{prop}\label{prop: invariance}
Suppose that $(\Omega,\mu)$ and $(\Omega',\mu')$ both admit the weighted Bergman metrics $g_{\Omega,\mu}$ and $g_{\Omega',\mu'}$ respectively.
Let $F:\Omega\rightarrow\Omega'$ be a biholomorphism.
If $\mu$ and $\mu'$ satisfy the relation:
$$
\mu'(F(z))=|h(z)|^2\mu(z)
$$
for some non-vanishing holomorphic function $h$ of $\Omega$,
then $F$ is an isometry with respect to the weighted Bergman metrics, i.e.,
$$
g_{\Omega,\mu}=F^*g_{\Omega',\mu'}.
$$
\end{prop}

\begin{proof}
The transformation formula (\ref{eqn: transformation}) for the weighted case implies that
$$
K_{\Omega,\mu}(z)=K_{\Omega',\mu'} (F(z))|\mathcal{J}(F(z)) h(z)|^2.
$$
By taking logarithm and mixed derivatives to both sides, we can show that
$$
\frac{\partial^2}{\partial z_j\partial\overline{z_k}}\log K_{\Omega,\mu}(z)
=
\sum^n_{l,m=1}\frac{\partial^2}{\partial \zeta_l\partial\overline{\zeta_m}}\log K_{\Omega',\mu'}(F(z))\frac{\partial F_l(z)}{\partial z_j}\overline{\left(\frac{\partial F_m(z)}{\partial z_k}\right)}.
$$
This implies that for any non-zero vector $X\in\mathbb{C}^n$, we have
$$
g_{\Omega,\mu}(z,X)
=
g_{\Omega',\mu'}(F(z),dF(X)),
$$
as we required.
\end{proof}

\begin{defn}\label{def: assignment}
Let $\mathcal{D}$ be a collection of domains in $\mathbb{C}^n$ and $\mathcal{M}$ be an assignment of an admissible weight function $\mu_{\Omega}:=\mathcal{M}(\Omega)$ to each domain $\Omega\in\mathcal{D}$.
We will call the weighted Bergman kernel $K_{\Omega,\mu_{\Omega}}$ {\it “$\mathcal{M}$-Bergman kernel"} of $\Omega$.\\
Suppose that for any domain $\Omega\in\mathcal{D}$, $g_{\mu_{\Omega}}$ is positive-definite.
In this case, we will call the weighted Bergman metric $g_{\Omega,\mu_{\Omega}}$
{\it “$\mathcal{M}$-Bergman metric"} of $\Omega$.
\end{defn}

\begin{defn}\label{def: invariant assignment}
We say that the assignment $\mathcal{M}$ is {\it invariant} (under biholomorphisms) if for any biholomorphism $F$ of $\Omega\in\mathcal{D}$, it satisfies that
$$
\mu_{F(\Omega)}\circ F=|h_F|^{2}\mu_{\Omega}
$$
for some non-vanishing holomorphic function $h_F$ (depending on $F$) of $\Omega$.
We say that an invariant assignment $\mathcal{M}$ is {\it canonical} (of level $m\in\mathbb{N}^+$) if it satisfies that
$$
\mu_{F(\Omega)}\circ F=|\mathcal{J}(F)|^{2(m-1)}\mu_{\Omega}.
$$

\end{defn}

\begin{rmk}
$\mathcal{M}$ and $h_F$ can be considered as a hermitian metric and the transition function of a trivial line bundle $L$ over $\Omega$, respectively.
In the case that $L$ is the canonical line bundle, the transition function is $h_F=\mathcal{J}(F)$.
\end{rmk}

As a corollary of Proposition \ref{prop: transformation} and Proposition \ref{prop: invariance}, we obtain the following

\begin{thm}\label{thm: invariance}
If an assignment $\mathcal{M}$ is invariant, the $\mathcal{M}$-Bergman metric is invariant under biholomorphisms, i.e.,
$$
g_{\Omega,\mu_{\Omega}}=F^*g_{F(\Omega),\mu_{F(\Omega)}}.
$$
If an invariant assignment $\mathcal{M}$ is canonical of level $m$, the $\mathcal{M}$-Bergman kernel satisfies the following transformation formula:
$$
K_{\Omega,\mu_{\Omega}}
=
K_{F(\Omega),\mu_{F(\Omega)}}|\mathcal{J}(F)|^{2m}
$$
so that the corresponding volume form is biholomorphically invariant:
$$
K_{\Omega,\mu_{\Omega}}^{\frac{1}{m}}(z)\lambda(z)
=
K_{F(\Omega),\mu_{F(\Omega)}}^{\frac{1}{m}}(w)\lambda(w),
$$
where $w=F(z)$.
\end{thm}

\begin{rmk}
If $\mathcal{M}$ is a canonical invariant assignment of level $m$, we will call the normalized function $K_{\Omega,\mu_{\Omega}}^{\frac{1}{m}}$ “$\mathcal{M}$-{\it normalized Bergman kernel}" of $\Omega$.
The level $m$ means the tensor power of the canonical line bundle.
\end{rmk}

\begin{ex}\label{ex1}
Let $\mathcal{D}^{\rm bp}$ be a collection of bounded pseudoconvex domains in $\mathbb{C}^n$.
By the famous work of Cheng-Yau\cite{cheng1980existence} and Mok-Yau\cite{mok1983completeness}, every $\Omega\in\mathcal{D}^{\rm bp}$ admits unique complete K\"{a}hler-Einstein metric $g^{\rm KE}_{\Omega}$.
Define an admissible assignment $\mathcal{M}^{\rm KE}$ by
$$
\mathcal{M}^{\rm KE}(\Omega):=e^{-\varphi^{\rm KE}_{\Omega}}=\frac{1}{\det\left(g^{\rm KE}_{\Omega}\right)}.
$$
By the uniqueness of the K\"ahler-Einstein metric and the volume form:
$$
\det\left(g^{\rm KE}_{\Omega}(z)\right)
=\det\left(g^{\rm KE}_{F(\Omega)}(F(z))\right)|\mathcal{J}(F(z))|^2,
$$
Proposition \ref{prop: invariance} and Theorem \ref{thm: invariance} imply that $\mathcal{M}^{\rm KE}$-Bergman metric is invariant under biholomorphisms.
\end{ex}

\begin{ex}\label{ex2}
Let $\Omega$ be a bounded domain in $\mathbb{C}^n$.
Then the diagonal part of the classic Bergman kernel function $K_{\Omega}$ is a positive smooth strictly plurisubharmonic function.
Let $\mathcal{D}^{\rm b}$ be a collection of bounded domains in $\mathbb{C}^n$.
Define an admissible assignment $\mathcal{M}^{\rm B}$ by
$$
\mathcal{M}^{\rm B}(\Omega):=\frac{1}{K_{\Omega}}.
$$
By the transformation formula for Bergman kernels:
$$
K_{\Omega}(z,z)=
K_{F(\Omega)}(F(z),F(z))|\mathcal{J}(F(z))|^2,
$$
Proposition \ref{prop: invariance} and Theorem \ref{thm: invariance} imply that $\mathcal{M}^{\rm B}$-Bergman metric is invariant under biholomorphisms.
\end{ex}

\newpage

\section{Estimates of the weighted Bergman kernels and metrics}
In \cite{tian1990}, Tian constructed an asymptotic expansion of the sequence of the weighted Bergman kernels and metrics for canonically polarized manifolds.
Later, this result is improved by Ruan, Zelditch, Lu, and so on \cite{ruan1998canonical,zelditch1998szego,lu2000lower}.
This is called the Tian-Yau-Zelditch expansion.
For the proof, Tian used the ``peaked section" method, based on the standard $\overline{\partial}$-estimates for complete manifolds.

For domains in $\mathbb{C}^n$, this method can be understood in terms of the weighted version of the Bergman minimum integral method and H\"{o}rmander's classic $L^2-\overline{\partial}$ theorem for domains.
In this section, we present a straightforward proof of a  domain version of the Tian-Yau-Zelditch expansion for the weighted Bergman kernels, metrics, and curvatures.

\subsection{Estimates of the minimum integrals}
Let $\Omega$ be a pseudoconvex domain in $\mathbb{C}^n$ and $\varphi$ be a smooth strictly plurisubharmonic function on $\Omega$.
Choose an weight function $\mu:=e^{-\varphi}$.
Fix a point $p\in\Omega$ and a vector $X\in\mathbb{C}^n$.
Consider the minimum integrals of the weighted Bergman kernels and denote the minimizers of $\mathcal{E}^j_{\mu}(\Omega)\subset\mathcal{A}^2_{\mu}(\Omega)$ by $v^j$, i.e.,
$$
I^j_{\mu}:=\inf_{u\in \mathcal{E}^j_{\mu}(\Omega)}\|u\|_{\Omega,\mu}^2
=\|v^j\|^2_{\Omega,\mu}.
$$
Since it is hard to obtain the explicit minimizer $v^j\in\mathcal{E}^j_{\mu}(\Omega)$ in general, we will approximate it by another function via the following slight variant of H\"{o}rmander's theorem to construct a holomorphic function in $\mathcal{E}^j_{\mu}(\Omega)$.

\begin{thm}[Theorem 5 in \cite{gallagher2017dimension}]\label{Hormander thm}
Let $\Omega$ be a pseudoconvex domain in $\mathbb{C}^n$, and let $\psi$ be a plurisubharmonic function in $\Omega$.
Suppose that $g$ is a $\overline{\partial}$-closed $C^{\infty}$-smooth $(0,1)$-form which has a compact support on some open subset $U\subset\Omega$.
If there exists a positive constant $C$ such that
$$
\psi(z)-C|z|^2
$$
is plurisubharmonic on $U$,
then there exists a $C^{\infty}$-smooth function $u\in L^2(\Omega,e^{-\psi})$ solving $\overline{\partial}u=g$ and satisfying
$$
\int_{\Omega}|u|^2e^{-\psi}d\lambda\leq
\frac{1}{C}\int_{U}|g|^2e^{-\psi}d\lambda.
$$
\end{thm}


Using Theorem \ref{Hormander thm} with a similar technique used in the localization lemma of minimum integrals (cf. \cite{greene2006function,kim1996boundary}), we can construct an explicit holomorphic function in $\mathcal{E}^j_{\mu}(\Omega)$ as follows:

Let $U,V$ be bounded neighborhoods of $p$ satisfying
$V\subset\subset U\subset\subset\Omega$.
Choose a cut-off function $\chi_U\in C^{\infty}_c(U)$ such that $\chi_U=1$ on $V$ and $0\leq\chi_U\leq1$ on $U$.
Let $j'$ be any integer $j'>j$.
Consider an upper-bounded function $\xi$ on $\Omega$ satisfying
$e^{\xi}\leq C_{\Omega}$ and
\begin{equation}\label{behavior near p}
e^{\xi(z)}:=O\left(|z-p|^{2(n+j')}\right)\ \ \
{\rm near\ } p.
\end{equation}
By the boundedness of $U$ and $V$, we may assume that on $U\backslash V$,
$$
|\overline{\partial}\chi_U|^2e^{-\xi}\leq C_{U\backslash V}
$$
for some positive constant $C_{U\backslash V}>0$.
\begin{prop}\label{upper estimates of minimum}
Suppose that
$
\psi:=\varphi+\xi
$
is plurisubharmonic on $\Omega$, and
$$
\psi(z)-C_U|z|^2
$$
is plurisubharmonic on $U$ for some positive constant $C_U$.
Let $\eta^j\in\mathcal{E}^j_{\mu}(W)$ be a function on some neighborhood $W$ of $p$ satisfying $U\subset W\subset\Omega$.
Then, there exists a holomorphic function $\widehat{u}^j\in\mathcal{E}^j_\mu(\Omega)$ satisfying
$$
\|\widehat{u}^j-\chi_U\cdot\eta^j\|^2_{\Omega,\mu}
\leq
C\|\eta^j\|^2_{U\backslash V,\mu},
$$
where $C$ depends only on $C_{\Omega},C_{U\backslash V},C_U$.
\end{prop}

\begin{proof}
Apply Theorem \ref{Hormander thm} with a $\overline{\partial}$-closed $C^{\infty}$-smooth $(0,1)$-form $g$, defined by
$$
g:=\overline{\partial}\left(\chi_U\cdot\eta^j\right).
$$
Then there exists a $C^{\infty}$-smooth function $u^j\in L^2(\Omega,e^{-\psi})$ solving $\overline{\partial}u=g$ and satisfying
\begin{equation}\label{l2 inequality}
\int_{\Omega}|u^j|^2e^{-\psi}d\lambda
\leq
\frac{1}{C_U}\int_{U\backslash V}|\overline{\partial}\chi_U|^2|\eta^j|^2e^{-\psi}d\lambda
\leq
\frac{C_{U\backslash V}}{C_U}\int_{U\backslash V}|\eta^j|^2e^{-\varphi}d\lambda.
\end{equation}
Since the right hand side integral is finite by assumption, the left hand side integral
$$
\int_{\Omega}|u^j|^2e^{-\psi}d\lambda
=
\int_{\Omega}\frac{|u^j|^2}{e^{\xi}}e^{-\varphi}d\lambda
$$
is also finite.
This with the property (\ref{behavior near p}) implies that
$
D^{\alpha}u^{j}(p)=0,
$
for all $\alpha=(\alpha_1,\ldots,\alpha_n)$ with $|\alpha|\leq j$.
Then the function $\widehat{u}^j$, defined by
$$
\widehat{u}^j:=\chi_U\cdot\eta^j-u^j
$$
satisfies that
$$
D^{\alpha}\widehat{u}^j(p)=D^{\alpha}\eta^j(p),
$$
for all $\alpha=(\alpha_1,\ldots,\alpha_n)$ with $|\alpha|\leq j$.
Since $e^{\xi}\leq C_{\Omega}$, the $L^2$-estimate (\ref{l2 inequality}) implies that
\begin{equation}
\|\widehat{u}^j-\chi_U\cdot\eta^j\|^2_{\Omega,\mu}
=
\|u^j\|_{\Omega,\mu}^2
\leq
\frac{C_{\Omega}C_{U\backslash V}}{C_U}\|\eta^j\|_{U\backslash V,\mu}^2,
\end{equation}
as we required.
\end{proof}
\begin{rmk}
In the localization lemma of minimum integrals for domains (cf. \cite{greene2006function,kim1996boundary}), the plurisubharmonic function
$
\xi(z):=\log|z-p|^{2(n+j')}
$
is usually applied
so that the constant $C$ depends on the distance from $p$ to boundary of $\Omega$.
However, we want to prove results for not only the pointwise convergence but also the uniform convergence in this paper.
Therefore, we will use a modified function $\xi$ by an appropriate cut-off function, following Tian's idea in \cite{tian1990}.
\end{rmk}
Note that since $v^j$ is the minimizer of $\mathcal{E}^j_{\mu}(\Omega)\subset\mathcal{A}^2_{\mu}(\Omega)$, we have an upper bound estimate for the minimum integral:
$$
I^j_{\mu}:
=
\|v^j\|^2_{\Omega,\mu}
\leq
\|\widehat{u}^j\|^2_{\Omega,\mu}
\leq
\|\eta^j\|^2_{U,\mu}
+
C\|\eta^j\|_{U\backslash V,\mu}^2
.
$$
For a sharp estimate of $\|\eta^j\|_{\mu}^2$, we will change the coordinates locally so that $\mu=e^{-\varphi}$ has good representation, using the following finite order approximation of the Bochner coordinates (or K-coordinates).
For later uses, we present a proof with details here.

\begin{lem}\label{Bochner lemma}
There exist a neighborhood $W$ of $p$ and an injective holomorphic map $f:W\rightarrow B^n(r_p)$ with $f(p)=0$ such that for a holomorphic function $h:W\rightarrow\mathbb{C}$, the Taylor expansion of the function
$$
{\Phi}:=(\varphi-2{\rm Re}(h))\circ f^{-1} 
$$
at the origin in the new coordinate system $w=(w_1,\ldots.w_n)=f(z)$ satisfies that
$$
\Phi(w)=|w|^2+\frac{1}{4}\sum_{i,j,k,l=1}^n\Phi_{i\overline{j}k\overline{l}}(0)w_i\overline{w_j}w_k\overline{w_l}+O(|w|^5),
$$
where $O(|w|^5)$ denotes terms which are at least quintic in $w,\overline{w}$-variables.
Moreover, the radius $r_p$ only depends on values
$$
\{D_z^{\alpha}\varphi_{j\overline{k}}(p)\},
$$
for all $\alpha=(\alpha_1,\ldots,\alpha_n)$ with $0\leq|\alpha|\leq 2$ and $j,k=1,\ldots,n$.
\end{lem}

\begin{proof}
Let $\zeta:=z-p$.
The Taylor expansion of $\varphi$ at $p$ is
\begin{align*}
\varphi(z)=\ & 
h(z)+\overline{h(z)}
+\sum_{i,j=1}^n\varphi_{i\overline{j}}(p)\zeta_i\overline{\zeta_j}
+\sum_{k=1}^n\left(\zeta_k\overline{h_k(z)}+\overline{\zeta_k}h_k(z)\right)\\
&+\sum_{i,j,k,l=1}^n\frac{1}{4}\varphi_{i\overline{j}k\overline{l}}(p)\zeta_i\overline{\zeta_j}\zeta_k\overline{\zeta_l}+O(|z-p|^5),
\end{align*}
where 
$$
h(z):=
\frac{1}{2}\varphi(p)+\sum_{1\leq|\alpha|\leq4}
\frac{D^{\alpha}_z\varphi(p)}{|\alpha|!}(z-p)^{\alpha}
{\rm \ \ and \ \ }
h_k(z):=\sum_{2\leq|\alpha|\leq3}
\frac{D^{\alpha}_z\varphi_{\overline{k}}(p)}{|\alpha|!}(z-p)^{\alpha}.
$$
Then the function 
$$
\phi(z):=\varphi(z)-h(z)-\overline{h(z)}
$$
satisfies that
for all $\alpha,\beta$ with $0\leq|\alpha|\leq4,\  1\leq|\beta|\leq3$, we have
\begin{equation}\label{vanishing terms}
  D^{\alpha}_z\phi(p)=0,
  {\rm \ \ and \ \ }
  D^{\beta}_z\phi_{\overline{k}}(p)=D^{\beta}_z\varphi_{\overline{k}}(p).
\end{equation}
Define $w=f(z)=(f_1(z),\ldots,f_n(z))$ by
\begin{equation}\label{def: bochner}
f_j(z)
:=\sum_{k=1}^n\sqrt{\varphi}^{\overline{k}j}(p)\left\{\sum_{1\leq|\alpha|\leq3}
\frac{D^{\alpha}_z\varphi_{\overline{k}}(p)}{|\alpha|!}(z-p)^{\alpha}\right\},
\end{equation}
where $\left(\sqrt{\varphi}^{\overline{k}j}(p)\right)$ is the matrix whose square is the inverse matrix of $\left({\varphi}_{j\overline{k}}(p)\right)$.
Consider the Taylor coefficients of the function
$$
{\Phi}(w):=(\varphi-2{\rm Re}(h))\circ f^{-1}(w)
=\phi\circ f^{-1}(w)
=\phi(z).
$$
Then a direct computation with (\ref{vanishing terms}) and (\ref{def: bochner}) implies that
$$
\Phi_{i\overline{j}}(0)=\delta_{ij},\ \
D^{\alpha}_w\Phi(p)=0,
{\rm \ \ and \ \ }
D^{\beta}_w\Phi_{\overline{k}}(p)=D^{\beta}_{\overline{w}}\Phi_{k}(p)=0,
$$
for all $\alpha,\beta$ with $0\leq|\alpha|\leq4,\  1\leq|\beta|\leq3$, as we required.
\end{proof}

\begin{rmk}
One can check that the coordinates $w=(w_1,\ldots.w_n)=f(z)$ are K\"ahler normal coordinates of the K\"ahler metric $i\partial\overline{\partial}\Phi=i\partial\overline{\partial}\varphi$.
If the given strictly plurisubharmonic function $\varphi$ is real-analytic, we can construct a better coordinate system in the sense that $D_w^{\alpha}\Phi_{j\overline{k}}(0)=0$
for all $\alpha$ with $|\alpha|\geq 1$.
This is called the {\it Bochner normal coordinates} (or {\it K-coordinates}), unique up to unitary linear transformations.
In fact, we can express the coordinate transform explicitly in terms of the polarization $\varphi(z,\overline{\zeta})$ of the given analytic function $\varphi(z)=\varphi(z,\overline{z})$:
$$
w_j(z)=\sum_{k=1}^n\sqrt{\varphi}^{\overline{k}j}(p)\left\{\frac{\partial}{\partial \overline{\zeta_k}}\Big|_{\zeta=p}\varphi(z,\overline{\zeta})-\frac{\partial}{\partial \overline{\zeta_k}}\Big|_{\zeta=p}\varphi(\zeta,\overline{\zeta})\right\}.
$$
In the case that the potential function $\varphi$ is the logarithm of the classic Bergman kernel function, the Bochner normal coordinates are called the {\it Bergman representative coordinates} (See \cite{yoo2017}).
\end{rmk}

Since $(w_1,\ldots,w_n)$ is a holomorphic normal coordinate system of the K\"ahler metric $g_{\varphi}:=i\partial\overline{\partial}\varphi=i\partial\overline{\partial}\Phi$ and $\Phi$ is a potential function,
one can check that
$$\label{rem: curvature tensor}
-\Phi_{i\overline{j}k\overline{l}}(0)
=
\sum^n_{\alpha,\beta,\gamma,\sigma=1}
\left(
-\varphi_{\alpha\overline{\beta}\gamma\overline{\sigma}}+\varphi^{\mu\overline{\nu}}\varphi_{\alpha\overline{\nu}\gamma}\varphi_{\mu\overline{\beta}\overline{\sigma}}
\right)
\frac{\partial z_\alpha}{\partial w_i}
\frac{\partial \overline{z_\beta}}{\partial \overline{w_j}}
\frac{\partial z_\gamma}{\partial w_k}
\frac{\partial \overline{z_\sigma}}{\partial \overline{w_l}}\Big|_{z=p},
$$
are the local expressions of the curvature $4$-tensor, i.e.,
$$
R_{i\overline{j}k\overline{l}}(w)
dw_i\otimes d\overline{w_j}\otimes dw_k\otimes d\overline{w_l}
=
R_{\alpha\overline{\beta}\gamma\overline{\sigma}}(z)
dz_{\alpha}\otimes d\overline{z_{\beta}}\otimes dz_{\gamma}\otimes d\overline{z_{\sigma}},
$$
where $R_{i\overline{j}k\overline{l}}(w):=-\Phi_{i\overline{j}k\overline{l}}(w)$.
Denote the Ricci and holomorphic sectional curvature at $p$ along $X$ by $S_{\varphi}(p),R_{\varphi}(p;X)$ and $H_{\varphi}(p;X)$, respectively.
For instance, if $\frac{\partial}{\partial w_1}\big|_0=df(X|_p)$, then
$$
H_{\varphi}(p;X):=R_{1\overline{1}1\overline{1}}(p)
=-\Phi_{1\overline{1}1\overline{1}}(0),
$$
$$
R_{\varphi}(p;X)
:=\sum_{i=1}^nR_{1\overline{1}i\overline{i}}(p)
=-\sum_{i=1}^n\Phi_{1\overline{1}i\overline{i}}(0).
$$
The scalar curvature is given by
$$
S_{\varphi}(p)
:=
\sum_{i,j=1}^nR_{i\overline{i}j\overline{j}}(0)
=
-\sum_{i,j=1}^n\Phi_{i\overline{i}j\overline{j}}(0).
$$

\subsection{Tian's sequence of Bergman kernels and metrics}
Let $\Omega$ be a pseudoconvex domain in $\mathbb{C}^n$ and $\varphi$ be a smooth strictly plurisubharmonic function on $\Omega$.
Consider a sequence of weight functions $\mu_{m+1}:=e^{-m\varphi}$ for non-negative integers $m\geq0$, and the corresponding weighted Bergman spaces
$$
\mathcal{A}^2_{\mu_{m+1}}(\Omega)
=
\mathcal{A}^2(\Omega,e^{-m\varphi}).
$$
Fix a point $p\in\Omega$ and a vector $X\in\mathbb{C}^n$.
For $j=0,1,2$, consider the minimum integrals of the weighted Bergman kernels and denote the minimizers of $\mathcal{E}^j_{\mu_{m+1}}(\Omega)$ by $v^j_{m+1}$, i.e.,
$$
I^j_{\mu_{m+1}}:=\inf_{u\in \mathcal{E}^j_{\mu_{m+1}}}\|u\|_{\mu_{m+1}}^2
=\|v^j_{m+1}\|^2_{\mu_{m+1}}.
$$
For an upper estimate of the minimum integral, we will apply Propostion \ref{upper estimates of minimum} with some suitable function $\eta^j=\eta^j_{m+1}\in\mathcal{E}^j_{\mu_{m+1}}(W)$, defined on some neighborhood $W\subset\Omega$ of $p$.

\begin{prop}\label{prop: asymptotic formula}
There exists $m_0>0$ such that for all $m\geq m_0$, there is a holomorphic function $\widehat{u}^j_{m+1}\in\mathcal{E}^j_{m+1}(\Omega)$ satisfying
\begin{equation}\label{peak eqn}
\|\widehat{u}^j_{m+1}\|^2_{\mu_{m+1}}
=\frac{e^{-m\varphi(p)}}{\det{\left(\varphi_{k\overline{l}}(p)\right)}}
\frac{1}{\left(g_{\varphi}(p;X)\right)^j}
\cdot\frac{1}{j!m^j}
\left(\frac{\pi}{m}\right)^n
\left(1+\frac{c_j}{m}+O\Big(\frac{1}{m^{2}}\Big)\right),
\end{equation}
and the constants $c_j$ are given by
$$
c_0=\frac{S_{\varphi}(p)}{2},\
c_1=\frac{S_{\varphi}(p)+2R_{\varphi}(p;X)}{2},\
c_2=\frac{S_{\varphi}(p)+4R_{\varphi}(p;X)+H_{\varphi}(p;X)}{2},
$$
where $g_{\varphi}(p;X)$ is the length square of $X$ with respect to the K\"ahler metric $g_{\varphi}:=i\partial\overline{\partial}\varphi$.
\end{prop}

\begin{proof}
Note that by definition, $u\in\mathcal{E}^j_{\mu_{m+1}}(\Omega)$ implies that $D^j_{cX}\left(\frac{1}{c^j}u\right)=1$ for any constant $c>0$.
Therefore, we may assume that by normalizing $X$, 
$$
g_{\varphi}(p;X)=|X|^{2}_{i\partial\overline{\partial}
\varphi}=1.
$$
Let $f$ be the holomorphic mapping, and let $W$ be the neighborhood of $p$ in Lemma \ref{Bochner lemma}.
Consider neighborhoods $V_m\subset U_m\subset W$ of $p$ such that
$$
f(V_m)=B^n(r_m/2)\subset f(U_m)=B^n(r_m)\subset f(W)=B^n(r_p).
$$
The radius $r_m$ will be chosen later so that $r_m\searrow0$ as $m\rightarrow\infty$.
Define a holomorphic function $\eta^j_{m+1}$ on $W$ by
$$
\eta^j_{m+1}(z):=\frac{1}{j!e^{mh(p)}\mathcal{J}(f(p))}(f_1(z))^j e^{mh(z)}\mathcal{J}(f(z)),
$$
where $h$ and $f(z)=(f_1(z),\ldots,f_n(z))$ are the holomorphisms in Lemma \ref{Bochner lemma}.
We may further assume that $df_p(X)$ and $\frac{\partial}{\partial w_1}$ are parallel by modifying $f$ via an unitary linear transform.
Then 
$$
\eta^j_{m+1}\in\mathcal{E}^j_{\mu_{m+1}}(W).
$$
Let $\chi:[0,\infty)\rightarrow[0,1]$ be a cut-off function satisfying $\chi(t)=1$ for $t\leq\frac{1}{2}$, $\chi(t)=0$ for $t\geq1$, $0\leq-\chi'(t)\leq4$, and $|\chi''(t)|\leq8$.
Then the function
$$
\chi_{U_m}(z):=\chi(|f(z)|/r_m)\in C^{\infty}_c({U_m})
$$
satisfies that $\chi_{U_m}=1$ on ${V_m}$ and $0\leq\chi_{U_m}\leq1$ on ${U_m}$.
Define a non-positive function $\xi_m$ on $\Omega$ by
$$
\xi_m:=(n+j')\chi_{U_m}\cdot\log(|f|^2/r^2_{m}).
$$
Then the definition (\ref{def: bochner}) implies that there exists a positive constant $C_{\varphi}(p)$, depending only derivatives of $\varphi_{j\overline{k}}(p)$, satisfying
$$
|f(z)|^2=|w|^2\leq C_{\varphi}(p)|z-p|^2.
$$
A direct computation shows that for a constant $C'=C'_{\varphi}(p)>0$, independent of $m$,
$$
\xi_m+\frac{C'(n+j')}{r_m^2}|z|^2
$$
is plurisubharmonic.
Consider the function
$$
\psi_m:=m\varphi+\xi_m
=
\left(m\varphi-\frac{C'(n+j')}{r_m^2}|z|^2\right)+\left(\xi_m+\frac{C'(n+j')}{r_m^2}|z|^2\right).
$$
Note that since $\varphi$ is strictly plurisubharmonic on $\Omega$, there exists a positive constant $C''=C''(\varphi,W)$ such that
$
\varphi-C''|z|^2
$
is plurisubharmonic on $W\subset\subset\Omega$.
Then,
$$
\psi_m-\left(C''m-\frac{C'(n+j')}{r_m^2}\right)|z|^2
$$
is plurisubharmonic on $U_m\subset W$.
Now take $r_m^2:=\frac{(\log m)^2}{m}$.
For sufficiently large $m$, there exists a positive constant $C$, independent of $m$ such that
$$
C_{U_m}:=
C''m-\frac{C'(n+j')}{r_m^2}
=
m\left(C''-\frac{C'(n+j')}{(\log m)^2}\right)>\frac{m}{C}
$$
is positive, and $\psi_m$ is plurisubharmonic on $\Omega$.
Then, we can apply Proposition \ref{upper estimates of minimum} with $\eta^j_m$ and $\psi_m$.
Note that
$$
|\overline{\partial}\chi_{U_m}|^2e^{-\xi_m}\leq Cr_m
$$ 
on $U_m\backslash V_m$.
Proposition \ref{upper estimates of minimum} implies that there exists a holomorphic function 
$$
\widehat{u}^j_{m+1}:=\chi_{U_m}\cdot\eta^j_{m+1}-u^j_{m+1}
\in\mathcal{E}^j_{\mu_{m+1}}(\Omega)
$$
satisfying
$$
\|\widehat{u}^j_{m+1}\|^2_{\Omega,\mu_{m+1}}
=
\|\chi_{U_m}\cdot\eta^j_{m+1}\|^2_{\Omega,\mu_{m+1}}
-2{\rm Re}\langle \chi_{U_m}\cdot\eta^j_{m+1},u^j_{m+1}\rangle_{\Omega,\mu_{m+1}}
+\|u^j_{m+1}\|^2_{\Omega,\mu_{m+1}},
$$
and
$$
\|u^j_{m+1}\|^2_{\Omega,\mu_{m+1}}
\leq
\frac{C}{(\log m)^2}\|\eta^j_{m+1}\|^2_{U_m\backslash V_m,\mu_{m+1}}.
$$
Note that for any neighborhood $U\subset W$,
\begin{align*}
\|\eta^j_{m+1}\|^2_{U,\mu_{m+1}}
&=
\frac{1}{|j!e^{mh(p)}\mathcal{J}(f(p))|^2}
\int_{U}\left| (f_1(z))^j\right|^2e^{-m(\varphi-2{\rm Re}(h))}\left|\mathcal{J}(f(z))\right|^2d\lambda(z)\\
&=
\frac{e^{-m\varphi(p)}}{(j!)^2\det{\left(\varphi_{k\overline{l}}(p)\right)}}
\int_{f(U)}\left| (w_1)^j\right|^2e^{-m{\Phi}(w)}d\lambda(w),
\end{align*}
since $2{\rm Re}(h(p))=\varphi(p)$ and $|\mathcal{J}(f(p))|^2=\det{\left(\varphi_{k\overline{l}}(p)\right)}$ (see the proof of Lemma \ref{Bochner lemma}).

On the other hand, the Taylor expansion of ${\Phi}$ implies that
\begin{equation}\label{eqn: taylor of potential}
e^{-m\Phi(w)}=e^{-m|w|^2}\left(1-\frac{m}{4}\sum_{i,j,k,l=1}^n\Phi_{i\overline{j}k\overline{l}}(0)w_i\overline{w_j}w_k\overline{w_l}-O(m|w|^5)\right).    
\end{equation}
so that we have
$$
\int_{B^n(r_m)\backslash B^n(r_m/2)}\left| (w_1)^j\right|^2e^{-m{\Phi}(w)}d\lambda(w)
=
O\left(e^{-mr_m^2}\right)
=
O\left(m^{-\log m}\right)
=
O\left(m^{-r}\right),
$$
for any $r>0$.
This implies that
$$
\|\widehat{u}^j_{m+1}\|^2_{\Omega,\mu_{m+1}}
=
\frac{e^{-m\varphi(p)}}{(j!)^2\det{\left(\varphi_{k\overline{l}}(p)\right)}}
\left(\int_{B^n(r_m)}\left| \chi\left(\frac{|w|}{r_m}\right)(w_1)^j\right|^2e^{-m{\Phi}(w)}d\lambda(w)+O\left(m^{-r}\right)\right).
$$
Note that
$$
\int_{B^n(\frac{r_m}{2})}\left| (w_1)^j\right|^2e^{-m{\Phi}}d\lambda
\leq
\int_{B^n(r_m)}\left| \chi\left(\frac{|w|}{r_m}\right)(w_1)^j\right|^2e^{-m{\Phi}}d\lambda
\leq
\int_{B^n(r_m)}\left| (w_1)^j\right|^2e^{-m{\Phi}}d\lambda,
$$
and
$$
\int_{\mathbb{B}(0,r_m)}|w^{\alpha}|^2e^{-m|w|^2}d\lambda
=
\int_{\mathbb{C}^n}|w^{\alpha}|^2e^{-m|w|^2}d\lambda
+O(m^{-r}).
$$
Hence, we can use the following Laplace type integral:
$$
\left(\frac{m}{\pi}\right)^n
\int_{\mathbb{C}^n}|w^{\alpha}|^2e^{-m|w|^2}d\lambda
=
\frac{\alpha_1!\cdots\alpha_n!}{m^{\alpha_1+\cdots+\alpha_n}}
=:
\frac{\alpha!}{m^{|\alpha|}}.
$$
Therefore, with the condition $g_{\varphi}(p;X)=1$, we have
$$
\|\widehat{u}^j_{m+1}\|^2_{\mu_{m+1}}
=\frac{e^{-m\varphi(p)}}{\det{\left(\varphi_{k\overline{l}}(p)\right)}}
\frac{1}{j!m^j}
\left(\frac{\pi}{m}\right)^n
\left(1+\frac{c_j}{m}+O\Big(\frac{1}{m^{2}}\Big)\right),
$$
Finally, the conclusion follows from the below lemma.
\end{proof}

\begin{lem}
The constants are given by
$$
c_0=\frac{S_{\varphi}(p)}{2},\
c_1=\frac{S_{\varphi}(p)+2R_{\varphi}(p;X)}{2},\
c_2=\frac{S_{\varphi}(p)+4R_{\varphi}(p;X)+H_{\varphi}(p;X)}{2}.
$$
\end{lem}

\begin{proof}
Note that
$$
\frac{c_j}{m}
=
-j!m^{j}\frac{m}{4}
\left(\frac{m}{\pi}\right)^n\sum_{i,t,k,l=1}^n\left\{\Phi_{i\overline{t}k\overline{l}}(0)\int_{\mathbb{C}^n}|w_1^j|^2w_i\overline{w_t}w_k\overline{w_l}e^{-m|w|^2}\ d\lambda\right\}.
$$
For any multi-indices $\alpha\neq\beta$, we have
$
\int_{\mathbb{C}^n}w^{\alpha}\overline{w^{\beta}}e^{-m|w|^2}d\lambda=0.
$
This implies that
\begin{align*}
&\sum_{i,j,k,l=1}^n\int_{\mathbb{C}^n}\Phi_{i\overline{j}k\overline{l}}(0)w_i\overline{w_j}w_k\overline{w_l}\ e^{-m|w|^2}d\lambda\\
=&
\int_{\mathbb{C}^n}\sum_{i=1}^n
\Phi_{i\overline{i}i\overline{i}}(0)|w_i|^4
e^{-m|w|^2}d\lambda
+
\int_{\mathbb{C}^n}
\sum_{i\neq j} \left(\Phi_{i\overline{i}j\overline{j}}(0)+\Phi_{i\overline{j}j\overline{i}}(0)\right)|w_i|^2|w_j|^2
e^{-m|w|^2}d\lambda\\
=&
\left(\frac{\pi}{m}\right)^n\frac{1}{m^{2}}
\left\{
2\sum_{i=1}^n \Phi_{i\overline{i}i\overline{i}}(0)
+
\sum_{i\neq j} \left(\Phi_{i\overline{i}j\overline{j}}(0)+\Phi_{i\overline{j}j\overline{i}}(0)\right)
\right\}\\
=&
\left(\frac{\pi}{m}\right)^n\frac{2}{m^{2}}
\sum_{i,j}\Phi_{i\overline{i}j\overline{j}}(0)
=
-\left(\frac{\pi}{m}\right)^n\frac{2}{m^{2}}S_{\varphi}(p)
.
\end{align*}
Therefore, $c_0=S_{\varphi}(p)/2$.
The constant $c_1$ can be computed by
\begin{align*}
&\sum_{i,j,k,l=1}^n\int_{\mathbb{C}^n}\Phi_{i\overline{j}k\overline{l}}(0)|w_1|^2w_i\overline{w_j}w_k\overline{w_l}\ e^{-m|w|^2}d\lambda\\
=&
\int_{\mathbb{C}^n}
\Phi_{1\overline{1}1\overline{1}}(0)|w_1|^6
e^{-m|w|^2}d\lambda
+
\int_{\mathbb{C}^n}
\sum_{i\geq2}\Phi_{i\overline{i}i\overline{i}}(0)|w_1|^2|w_i|^4
e^{-m|w|^2}d\lambda\\
&+
\int_{\mathbb{C}^n}\sum_{i\geq2}\left(\Phi_{1\overline{1}i\overline{i}}(0)+\Phi_{i\overline{i}1\overline{1}}(0)+\Phi_{1\overline{i}i\overline{1}}(0)+\Phi_{i\overline{1}1\overline{i}}(0)\right)|w_1|^4|w_i|^2
e^{-m|w|^2}d\lambda\\
&+
\int_{\mathbb{C}^n}
\sum_{i\neq j\geq2}\left(\Phi_{i\overline{i}j\overline{j}}(0)+\Phi_{i\overline{j}j\overline{i}}(0)\right)|w_1|^2|w_i|^2|w_j|^2
e^{-m|w|^2}d\lambda
\\
=&
\left(\frac{\pi}{m}\right)^n\frac{1}{m^3}
\left\{
6\Phi_{1\overline{1}1\overline{1}}(0)
+
2\sum_{i\geq2}\Phi_{i\overline{i}i\overline{i}}(0)
+
8\sum_{i\geq2}\Phi_{1\overline{1}i\overline{i}}(0)
+
2\sum_{i\neq j\geq2}\Phi_{i\overline{i}j\overline{j}}(0)
\right\}
\\
=&
\left(\frac{\pi}{m}\right)^n\frac{1}{m^3}
\left\{
4\Phi_{1\overline{1}1\overline{1}}(0)
+
2\sum_{i=1}^n\Phi_{i\overline{i}i\overline{i}}(0)
+
4\sum_{i\geq2}\Phi_{1\overline{1}i\overline{i}}(0)
+
2\sum_{i\neq j}\Phi_{i\overline{i}j\overline{j}}(0)
\right\}
\\
=&
\left(\frac{\pi}{m}\right)^n\frac{1}{m^3}
\left\{
4\sum_{i=1}^n\Phi_{1\overline{1}i\overline{i}}(0)
+
2\sum_{i,j}\Phi_{i\overline{i}j\overline{j}}(0)
\right\}
=-\left(\frac{\pi}{m}\right)^n\frac{1}{m^3}
\left(
4R_{\varphi}(p;X)+2S_{\varphi}(p)
\right),
\end{align*}
Similarly, one can show that
\begin{align*}
&\sum_{i,j,k,l=1}^n\int_{\mathbb{C}^n}\Phi_{i\overline{j}k\overline{l}}(0)|w_1|^4w_i\overline{w_j}w_k\overline{w_l}\ e^{-m|w|^2}d\lambda\\
=&
-\left(\frac{\pi}{m}\right)^n\frac{1}{m^4}
\left(
4H_{\varphi}(p;X)+16R_{\varphi}(p;X)+4S_{\varphi}(p)
\right).
\end{align*}
\end{proof}

Define sub-spaces of the weighted Bergman space $\mathcal{H}_m:=\mathcal{A}^2_{\mu_m}(\Omega)$ of co-dimension $1$:
\begin{align*}
\mathcal{H}^0_m&:=\left\{u\in\mathcal{H}_m:u(p)=0\right\},\\
\mathcal{H}^1_m&:=\left\{u\in\mathcal{H}^0_m: D_Xu(p)=0 \right\},\\
\mathcal{H}^2_m&:=\left\{u\in\mathcal{H}^1_m: D_XD_Xu(p)=0 \right\}.
\end{align*}

The following proposition shows that the function $\widehat{u}^j_m$ is \emph{asymptotically orthogonal} to the subspace $\mathcal{H}^j_m$ as $m\rightarrow\infty$.

\begin{prop}\label{prop: orthogonality}
For all $v\in\mathcal{H}^j_{m+1}$, we have
\begin{equation}\label{eqn: orthogonality}
|\langle \widehat{u}^j_{m+1},v\rangle_{\mu_{m+1}}|
=O\Big(\frac{1}{m}\Big)\|\widehat{u}^j_{m+1}\|_{\mu_{m+1}}\|v\|_{\mu_{m+1}}.
\end{equation}
\end{prop}
\begin{proof}
Let $v$ be a function in $\mathcal{H}^j_{m+1}$.
On $W$, we can represent $v$ as
$$
v(z)=
\frac{v(z)}{e^{mh(z)}\mathcal{J}(f(z))}e^{mh(z)}\mathcal{J}(f(z))
=:
\widetilde{v}(f(z))\ e^{mh(z)}\mathcal{J}(f(z)).
$$
As in Proposition \ref{prop: asymptotic formula}, we may assume that $df_p(X)=\frac{\partial}{\partial w_1}$.
Since $v\in\mathcal{H}^j_{m+1}$,
\begin{equation}\label{eqn: taylor}
\left(\frac{\partial}{\partial w_1}\right)^k\widetilde{v}(0)=0,
\ \ \ \ 
{\rm for\ all\ \ }
0\leq k\leq j.
\end{equation}
Recall that
$$
\widehat{u}^j_{m+1}:=\chi_{U_m}\cdot\eta^j_{m+1}-u^j_{m+1}
\in\mathcal{E}^j_{\mu_{m+1}}(\Omega)
$$
and
$$
\eta^j_{m+1}(z)=\frac{1}{j!e^{mh(p)}\mathcal{J}(f(p))}(f_1(z))^j e^{mh(z)}\mathcal{J}(f(z))
=:c(p,m)(f_1(z))^j e^{mh(z)}\mathcal{J}(f(z)).
$$
Using the Taylor expansion (\ref{eqn: taylor of potential}), we can show that 
\begin{align*}
\langle\widehat{u}^j_{m+1},v\rangle_{\mu_{m+1}}
&=
\langle\chi_{U_m}\cdot\eta^j_{m+1}-u^j_{m+1},v\rangle_{\Omega,\mu_{m+1}}
\\
&=
\langle\chi_{U_m}\cdot\eta^j_{m+1},v\rangle_{U_m,\mu_{m+1}}
-\langle u^j_{m+1},v\rangle_{\Omega,\mu_{m+1}}\\
&=
c(p,m)\int_{\mathbb{B}(0,r_m)}w_1^j\ \overline{\widetilde{v}(w)}e^{-m|w|^2}\left(1+O(m|w|^4)\right)d\lambda
+
O(m^{-r})\|v\|_{\mu_{m+1}}.
\end{align*}
Note that the condition (\ref{eqn: taylor}) implies that 
$$
\int_{\mathbb{B}(0,r_m)}w_1^j\ \overline{\widetilde{v}(w)}e^{-m|w|^2}d\lambda=0.
$$
Applying the Cauchy-Schwarz inequality to the integral, we obtain
\begin{align*}
&\Big|\int_{\mathbb{B}(0,r_m)}c(p,m)w_1^j\ \overline{\widetilde{v}(w)}e^{-m|w|^2}O(m|w|^4)d\lambda\Big|
\\
&\leq
C
\left(m^2
\int_{\mathbb{B}(0,r_m)}|c(p,m)|^2|w_1^j|^2|w|^{8}e^{-m{\Phi}(w)}d\lambda
\right)^{\frac{1}{2}}\|v\|_{\mu_{m+1}}
\\
&\leq
\frac{C}{m}
\|\widehat{u}^j_m\|_{\mu_{m+1}}\|v\|_{\mu_{m+1}},
\end{align*}
as we required.
\end{proof}

\begin{lem}\label{lem: asymptotic minimizer}
Let $v^j_{m+1}$ be the minimizers of $\mathcal{E}^j_{\mu_{m+1}}$, i.e.,
$I^j_{\mu_{m+1}}
=\|v^j_{m+1}\|^2_{\mu_{m+1}}$
for $j=0,1,2$.
Then, we have
$$
\left(1-O\left(\frac{1}{m^2}\right)\right)\|\widehat{u}^j_{m+1}\|^2_{\mu_{m+1}}
\leq
\|v^j_{m+1}\|^2_{\mu_{m+1}}
\leq
\|\widehat{u}^j_{m+1}\|^2_{\mu_{m+1}}.
$$
\end{lem} 
\begin{proof}
Choose an orthogonal basis $\{v^{k}_{m+1}\}_{k=0}^{\infty}$ of the weighted Bergman space $\mathcal{A}^2_{\mu_{m+1}}(\Omega)$ including the minimizers
$v^0_{m+1},v^1_{m+1},v^2_{m+1}$.
Consider an orthogonal decomposition:
$$
\widehat{u}^j_{m+1}
=\sum_{k=0}^ja_kv^k_{m+1}+\widetilde{v}_{m+1}^{j+1},
$$
where $\widetilde{v}_{m+1}^{j+1}\in\mathcal{H}^j_{m+1}$.
Since $v^j_m\in\mathcal{E}^j_{m}$, we have
$a_j=1$, and $a_k=0$ if $0\leq k< j$ so that
$$
\widehat{u}^j_{m+1}
=v^j_{m+1}+\widetilde{v}_{m+1}^{j+1},
$$
and $v^j_{m+1}$ is orthogonal to $\widetilde{v}_{m+1}^{j+1}$.
Proposition \ref{prop: orthogonality} implies that
$$
\|\widetilde{v}_{m+1}^{j+1}\|^2
=
\|\widehat{u}^j_{m+1}
-v^j_{m+1}\|^2
=
\langle
\widehat{u}^j_{m+1},\widehat{u}^j_{m+1}
-v^j_{m+1}
\rangle
\leq
\frac{C}{m}
\|\widehat{u}_{m+1}^{j}\|
\|\widetilde{v}_{m+1}^{j+1}\|
$$
so that
$$
\|\widetilde{v}_{m+1}^{j+1}\|^2
\leq
\frac{C}{m^2}
\|\widehat{u}_{m+1}^{j}\|^2.
$$
The conclusion follows from the following:
$$
\|v_{m+1}^{j}\|^2
=
\|\widehat{u}_{m+1}^{j}\|^2
-
\|\widetilde{v}_{m+1}^{j+1}\|^2
\geq
\left(1-\frac{C}{m^2}\right)\|\widehat{u}_{m+1}^{j}\|^2.
$$
\end{proof}

We have the following asymptotic expansion for the sequence of the weighted Bergman kernels, metrics, and curvatures with respect to the weight $\mu_{m+1}:=e^{-m\varphi}$.

\begin{thm}\label{thm: asymptototic expansion}
Let $\Omega$ be a pseudoconvex domain in $\mathbb{C}^n$ with a smooth strictly plurisubharmonic function $\varphi$.
Suppose that $(\Omega,e^{-\varphi})$ admits the positive definite weighted Bergman metric.
Fix a point $p\in\Omega$ and a vector $X\in\mathbb{C}^n$.
There exists $m_0>0$ such that for all $m\geq m_0$, we have 
\begin{align*}
K_{\Omega,\mu_{m+1}}(p)
&=
\frac{\det{\left(\varphi_{k\overline{l}}(p)\right)}}{e^{-m\varphi(p)}}
\left(\frac{m}{\pi}\right)^n
\left(1-\frac{S_{\varphi}(p)}{2m}+O\Big(\frac{1}{m^{2}}\Big)\right),
\\
g_{\Omega,\mu_{m+1}}(p;X)
&=
mg_{\varphi}(p;X)
\left(1-\frac{R_{\varphi}(p;X)}{m}+O\Big(\frac{1}{m^{2}}\Big)\right),
\\
H_{\Omega,\mu_{m+1}}(p;X)
&=
\frac{H_{\varphi}(p;X)}{m}+O\Big(\frac{1}{m^{2}}\Big).
\end{align*}
\end{thm}

\begin{proof}
Let $v^j_{m+1}$ be minimizers of $\mathcal{E}^j_{\mu_{m+1}}$, i.e.,
$I^j_{\mu_{m+1}}
=\|v^j_{m+1}\|^2_{\mu_{m+1}}$.
Proposition \ref{prop: asymptotic formula} and Lemma \ref{lem: asymptotic minimizer} imply that
\begin{equation*}
I^j_{\mu_{m+1}}
=\frac{e^{-m\varphi(p)}}{\det{\left(\varphi_{k\overline{l}}(p)\right)}}
\frac{1}{g_{\varphi}(p;X)^j}
\frac{1}{j!m^j}
\left(\frac{\pi}{m}\right)^n
\left(1+\frac{c_j}{m}+O\Big(\frac{1}{m^{2}}\Big)\right).   
\end{equation*}
Then, the conclusion follows from the Bergman-Fuks formula in Theorem \ref{thm: bergman-fuks}:
$$
 K_{\Omega,\mu_{m+1}}(p)=\frac{1}{I^0_{\mu_{m+1}}},\ \ \ \  
 g_{\Omega,\mu_{m+1}}(p;X)=\frac{I^0_{\mu_{m+1}}}{I^1_{\mu_{m+1}}},\ \ \ \ 
 H_{\Omega,\mu_{m+1}}(p;X)=2-\frac{(I^1_{\mu_{m+1}})^2}{I^2_{\mu_{m+1}}I^0_{\mu_{m+1}}}.
$$
\end{proof}

\begin{cor}\label{cor: asymptototic expansion}
For $\mu_{m}:=e^{-(m-1)\varphi}$, we have the following
\begin{align*}
K_{\Omega,\mu_m}(p)
&=
e^{m\varphi(p)}\frac{\det{\left(\varphi_{k\overline{l}}(p)\right)}}{e^{\varphi(p)}}
\left(\frac{m}{\pi}\right)^n
\left(1-\frac{S_{\varphi}(p)+2n}{2m}+O\Big(\frac{1}{m^{2}}\Big)\right),
\\
g_{\Omega,\mu_{m}}(p;X)
&=
mg_{\varphi}(p;X)
\left(1-\frac{R_{\varphi}(p;X)+1}{m}+O\Big(\frac{1}{m^{2}}\Big)\right),
\\
H_{\Omega,\mu_{m}}(p;X)
&=
\frac{H_{\varphi}(p;X)}{m}+O\Big(\frac{1}{m^{2}}\Big).
\end{align*}
\end{cor}

\begin{proof}
Use the following modification:
\begin{align*}
I^j_{\mu_m}
&=\frac{e^{-(m-1)\varphi(p)}}{\det{\left(\varphi_{k\overline{l}}(p)\right)}}
\frac{1}{g_{\varphi}(p;X)^j}
\frac{1}{j!(m-1)^j}
\left(\frac{\pi}{m-1}\right)^n
\left(1+\frac{c_j}{m-1}+O\Big(\frac{1}{(m-1)^{2}}\Big)\right)\\
&=\frac{e^{-(m-1)\varphi(p)}}{\det{\left(\varphi_{k\overline{l}}(p)\right)}}
\frac{1}{\left(g_{\varphi}(p;X)\right)^j}
\frac{1}{j!m^j}
\left(\frac{\pi}{m}\right)^n
\left(1+\frac{c_j+n+j}{m}+O\Big(\frac{1}{m^{2}}\Big)\right).
\end{align*}
\end{proof}

\begin{cor}\label{cor: convergence of normalized sequences}
As $m\rightarrow\infty$, we have the following pointwise convergences:
$$
\sqrt[m]{K_{\Omega,\mu_m}}(p)\rightarrow e^{\varphi(p)},\ \ \ \ 
\frac{1}{m}g_{\Omega,\mu_m}(p;X)\rightarrow g_{\varphi}(p;X),\ \ \ \
mH_{\Omega,\mu_m}(p;X)\rightarrow H_{\varphi}(p;X).
$$
\end{cor}

\begin{rmk}\label{rmk: uniform convergence}
The convergence speed at $p$ depends on the big $O$-terms in the previous Theorem, which are determined by the derivatives
$
\{D_z^{\alpha}\varphi_{j\overline{k}}(p)\}
$
of the metric $g_{\varphi}=i\partial\overline{\partial}\varphi$,
and the constant $C''$, the lower bound of eigenvalues of the levi form $i\partial\overline{\partial}\varphi$ on $W$.
\end{rmk}

In \cite{tian1990}, Tian proved the $C^2$-convergence of the sequence of weighted Bergman metrics ($C^4$-convergence of the sequence of the weighted Bergman kernels) for compact polarized manifolds.
Later, Ruan\cite{ruan1998canonical} and Zelditch\cite{zelditch1998szego}, respectively improved this result up to to $C^{\infty}$ level.
Since this is a local statement (pointwise convergence), the same result also holds for the non-compact case.
Therefore, the convergence in Corollary \ref{cor: convergence of normalized sequences} can be improved up to all $C^{k}$-derivatives of the weighted Bergman kernels.

\section{Convergence of invariant weighted Bergman sequences}

For compact manifolds, the pointwise convergence implies the uniform convergence.
More generally, it is known that the convergence of Tian's sequence is uniform if the given (possibly non-compact) manifold satisfies the conditions of the bounded geometry (cf. \cite{ma2015exponential}).
Bounded domains in $\mathbb{C}^n$ with the properties of {\it uniformly squeezing} \cite{yeung2009} (also called the {\it holomorphic homogeneous regular} \cite{liu2004canonical}) belong to this category.
The proof of the uniform convergence in this case is much simpler thanks to the existence of the global coordinates and the transformation formula for weighted Bergman kernels.

In this section, we will focus two sequences of the weighted Bergman kernels, developed by Tian and Tsuji respectively.
They are important examples in the sense that these sequences converge to the volume form of the unique K\"ahler-Einstein metric, and the corresponding weighted Bergman metrics are biholomorphically invariant.
Since we will prove the uniform convergence of these sequences on uniform squeezing domains, we first briefly review known related results.
\subsection{Uniform squeezing domains}

\begin{defn}
A (bounded) domain $\Omega\subset\mathbb{C}^n$ is called an {\it uniform squeezing domain} if for any point $p\in\Omega$, there exist $r\in(0,1]$ (independent of $p$) and a biholomorphism $F_p$ on $\Omega$ satisfying $F_p(p)=0$ and 
$$
\mathbb{B}^n(r)\subset \Omega_p\subset \mathbb{B}^n(1),
$$
where $\Omega_p:=F_p(\Omega)$.
The supremum of such $r\in(0,1]$ is called the {\it uniform squeezing number}.
\end{defn}

\begin{rmk}
There are many important examples of bounded domains satisfying the uniform squeezing property such as homogeneous domains, convex domains, strongly pseudoconvex domains (see \cite{yeung2009,deng2016properties,kim2016uniform}).
\end{rmk}

Recall that by the famous theorem by Cheng-Yau \cite{cheng1980existence} and Mok-Yau \cite{mok1983completeness}, every bounded pseudoconvex domain $\Omega$ admits the unique complete K\"{a}hler-Einstein metric 
$$
g^{\rm KE}_{\Omega}
=
g^{\rm KE}_{\Omega,\alpha\overline{\beta}}(z)dz_{\alpha}\otimes d\overline{z_{\beta}}
$$
satisfying
$$
R_{\alpha\overline{\beta}}=-g^{\rm KE}_{\Omega,\alpha\overline{\beta}},
$$
where
$R_{\alpha\overline{\beta}}:=-\frac{\partial^2}{\partial z_{\alpha}\partial\overline{z_{\beta}}}\log\det(g^{\rm KE}_{\Omega,\gamma\overline{\delta}})$
is the Ricci tensor of the metric $g^{\rm KE}_{\Omega}$.
Then a potential function of the K\"ahler-Einstein  metric, defined by
$$
\varphi^{\rm KE}_{\Omega}
:=
\log\det(g^{\rm KE}_{\Omega}),
$$
is smooth strictly plurisubharmonic on $\Omega$,
where $(g^{\rm KE}_{\Omega})$ denotes the matrix representation of the K\"ahler-Eintstein metric with respect to the standard Euclidean coordinates for $\mathbb{C}^n\supset\Omega$.

\begin{thm}[Yeung \cite{yeung2009}]\label{thm: Yeung}
Let $\Omega$ be a uniform squeezing domain in $\mathbb{C}^n$.
Then $\Omega$ is pseudoconvex so that admits the unique complete K\"{a}hler-Einstein metric $g^{\rm KE}_{\Omega}$.
Moreover, $(\Omega,g^{\rm KE}_{\Omega})$ has bounded geometry of infinite order in the sense that for every positive integer $k$, there exists a positive constant $C_k$ (independent of $p$) satisfying
$$
\|\varphi^{\rm KE}_{\Omega_p}\|_{C^k(\mathbb{B}^n(r/2))}
\leq
C_k,
$$
where
$$
{\varphi^{\rm KE}_{\Omega_p}(F_p(z))}
:=
\log\det(g^{\rm KE}_{\Omega_p}(F_p(z))
)
=
\log(\det(g^{\rm KE}_{\Omega}(z))/|\mathcal{J}(F_p(z))|^2).
$$
\end{thm}

\subsection{Tian's Bergman sequence related to KE metric}
Let $\mathcal{D}^{\rm bp}$ be a collection of bounded pseudoconvex domains in $\mathbb{C}^n$.
Define a sequence of admissible assignments $\mathcal{M}^{\rm KE}_m$ on $\mathcal{D}^{\rm bp}$ for $m\in\mathbb{N}_+$ by
$$
\mathcal{M}^{\rm KE}_m(\Omega)
:=
\mu^{\rm KE}_{\Omega,m}
:=
e^{-(m-1)\varphi^{\rm KE}_{\Omega}}
=
\frac{1}{{\det\left(g^{\rm KE}_{\Omega}\right)}^{(m-1)}},
$$
for $\Omega\in\mathcal{D}^{\rm bp}$. 
Denote the $\mathcal{M}^{\rm KE}_m$-Bergman kernel by $K^{\rm KE}_{\Omega,m}:=K_{\Omega,\mu^{\rm KE}_{\Omega,m}}$ and the $\mathcal{M}^{\rm KE}_m$-Bergman metric by $g^{\rm KE}_{\Omega,m}:=g_{\Omega,\mu^{\rm KE}_{\Omega,m}}$, and the curvature by $H^{\rm KE}_{\Omega,m}:=H_{\Omega,\mu^{\rm KE}_{\Omega,m}}$.

\begin{prop}
The assignment $\mathcal{M}^{\rm KE}_m$ is invariant and canonical of level $m$ so that the $\mathcal{M}^{\rm KE}_m$-Bergman metrics are invariant under biholomorphisms.
\end{prop}

\begin{proof}
Let $\Omega\in\mathcal{D}^{\rm bp}$ be a bounded pseudoconvex domain.
By the uniqueness of the K\"ahler-Einstein metric and the volume form, for any biholomorphism $F$, we have
$$
\det\left(g^{\rm KE}_{\Omega}(z)\right)^m
=
\det\left(g^{\rm KE}_{F(\Omega)}(F(z))\right)^m|\mathcal{J}(F(z))|^{2m}.
$$
This implies that
$$
\mu^{\rm KE}_{F(\Omega),m+1}(F(z))
=
\frac{1}{{\det(g^{\rm KE}_{F(\Omega)}(F(z)))}^m}
=
\frac{|\mathcal{J}(F(z))|^{2m}}{\det\left(g^{\rm KE}_{\Omega}(z)\right)^m}
=
|\mathcal{J}(F(z))|^{2m}
\mu^{\rm KE}_{\Omega,m+1}(z).
$$
Therefore, the assignment $\mathcal{M}^{\rm KE}_m$ is invariant and canonical of level $m$.
Theorem \ref{thm: invariance} imply that $\mathcal{M}^{\rm KE}_m$-Bergman metrics are invariant under biholomorphisms.
\end{proof}
Consider the sequence of $\mathcal{M}^{\rm KE}_m$-\emph{normalized} Bergman kernels, metrics, and curvatures:
$$
\widetilde{K}^{\rm KE}_{\Omega,m}:=\sqrt[m]{K^{\rm KE}_{\Omega,m}},\ \ \ \ 
\widetilde{g}^{\rm KE}_{\Omega,m}
:=\frac{1}{m}g^{\rm KE}_{\Omega,m},\ \ \ {\rm\ and\ } \ \ \
\widetilde{H}^{\rm KE}_{\Omega,m}
:=mH^{\rm KE}_{\Omega,m}.
$$
Let $\varphi:=\varphi^{\rm KE}_{\Omega}=\log\det(g^{\rm KE}_{\Omega})$.
Then 
$g_{\varphi}=i\partial\overline{\partial}\varphi=g^{\rm KE}_{\Omega}$
and $H_{\varphi}=H^{\rm KE}_{\Omega}:=H_{g^{\rm KE}_{\Omega}}$.
Therefore, Corollary \ref{cor: asymptototic expansion} and Corollary \ref{cor: convergence of normalized sequences} imply the pointwise convergences for the above sequences.
For uniform squeezing domains, in fact, the convergences are uniform.

\begin{thm}\label{thm: tian ke seq}
If $\Omega$ has the uniform squeezing property,
we have the following uniform convergences:
$$
    \widetilde{K}^{\rm KE}_{\Omega,m}
    \rightarrow \det\left(g^{\rm KE}_{\Omega}\right),\ \ \ 
    \widetilde{g}^{\rm KE}_{\Omega,m}
    \rightarrow g^{\rm KE}_{\Omega},\ \ \ 
    \widetilde{H}^{\rm KE}_{\Omega,m}
    \rightarrow H^{\rm KE}_{\Omega}:=H_{g^{\rm KE}_{\Omega}},
$$
as $m\rightarrow\infty$.
\end{thm}

\begin{proof}
The transformation formula for the uniform squeezing map $F_p$ implies that
$$
K^{\rm KE}_{\Omega,m}(p)
=
K^{\rm KE}_{\Omega_p,m}(0)
|\mathcal{J}(F_p(p))|^{2m}.
$$
Therefore, for arbitrary point $p\in\Omega$, we have
$$
\frac{K^{\rm KE}_{\Omega,m}(p)}{\det\left(g^{\rm KE}_{\Omega}(p)\right)^{m}}
=
\frac{K^{\rm KE}_{\Omega_p,m}(0)}{{\det(g^{\rm KE}_{\Omega_p}(0))}^{m}}.
$$
Apply Corollary \ref{cor: asymptototic expansion} to $\Omega_p$ with $\varphi^{\rm KE}_{\Omega_p}=\log\det(g^{\rm KE}_{\Omega_p})$ at the origin, we obtain
$$
\frac{K^{\rm KE}_{\Omega_p,m}(0)}{{\det(g^{\rm KE}_{\Omega_p}(0))}^{m}}
=
\frac{K^{\rm KE}_{\Omega_p,m}(0)}{e^{m\varphi^{\rm KE}_{\Omega_p}(0)}}
=
\left(\frac{m}{\pi}\right)^n
\left(1-\frac{n}{2m}+O\Big(\frac{1}{m^{2}}\Big)\right),
$$
since $\det{(\varphi^{\rm KE}_{\Omega_p}(0))}=e^{\varphi^{\rm KE}_{\Omega_p}(0)}$ and $S_{\varphi^{\rm KE}_{\Omega_p}}(0)=-n$ by the K\"ahler-Einstein condition.
Taking the $m$-th root to the above equation, we have
$$
\frac{\widetilde{K}^{\rm KE}_{\Omega,m}}{\det(g^{\rm KE}_{\Omega})}(p)
=
\left(\frac{m}{\pi}\right)^{\frac{n}{m}}
\left(1-\frac{n}{2m}+O\Big(\frac{1}{m^{2}}\Big)\right)^{\frac{1}{m}}.
$$
Similarly, Corollary \ref{cor: asymptototic expansion} implies that
$$
\widetilde{g}^{\rm KE}_{\Omega,m}=g^{\rm KE}_{\Omega}\left(1+O\left(\frac{1}{m^2}\right)\right),\ \ \ \ 
\widetilde{H}^{\rm KE}_{\Omega,m}=H^{\rm KE}_{\Omega}\left(1+O\left(\frac{1}{m}\right)\right)
$$
Finally, the conclusion follows from Remark \ref{rmk: uniform convergence} and Theorem \ref{thm: Yeung}.
\end{proof}


\subsection{Tsuji's iterative Bergman sequence}

Let $\Omega$ be a bounded domain in $\mathbb{C}^n$.
Define a sequence of weight functions $\mu^{\rm B}_{\Omega,m}$ for $m\geq1$ inductively as follows.
Set $\mu^{\rm B}_{\Omega,1}:=1_{\Omega}$.
For the given admissible weight $\mu^{\rm B}_{\Omega,m}$, consider the weighted Bergman space
$$
\mathcal{A}^2_{{\rm B}_m}(\Omega)
:=
\mathcal{A}^2_{\mu^{\rm B}_{\Omega,m}}(\Omega)
=
\mathcal{A}^2(\Omega,\mu^{\rm B}_{\Omega,m}).
$$
Denotes the corresponding weighted Bergman kernel of $\mathcal{A}^2_{{\rm B}_m}(\Omega)$ by
$$
K^{\rm B}_{\Omega,m}
:=
K_{\Omega,\mu^{\rm B}_{\Omega,m}}.
$$
Define the next admissible weight function by
$$
\mu^{\rm B}_{\Omega,m+1}
:=
\frac{1}{K^{\rm B}_{\Omega,m}}
=
\frac{1}{K_{\Omega,\mu^{\rm B}_{\Omega,m}}}.
$$
Inductively, consider the weighted Bergman space with respect to the above weight:
$$
\mathcal{A}^2_{{\rm B}_{m+1}}(\Omega)
:=
\mathcal{A}^2_{\mu^{\rm B}_{\Omega,m+1}}(\Omega)
=
\mathcal{A}^2(\Omega,\frac{1}{K^{\rm B}_{\Omega,m}}).
$$
Denotes the weighted Bergman metric of the kernel $K^{\rm B}_{\Omega,m}$ by
$$
g^{\rm B}_{\Omega,m}
:=
g_{\Omega,\mu^{\rm B}_{\Omega,m}}.
$$

Let $\mathcal{D}^{\rm b}$ be a collection of bounded domains in $\mathbb{C}^n$.
Define a sequence of admissible assignments $\mathcal{M}^{\rm B}_m$ on $\mathcal{D}^{\rm b}$ by
$$
\mathcal{M}^{\rm B}_m(\Omega)
:=
\mu^{\rm B}_{\Omega,m}.
$$

\begin{prop}\label{prop: invariance of Tsuji}
The assignment $\mathcal{M}^{\rm B}_m$ is invariant and canonical of level $m$ so that the $\mathcal{M}^{\rm B}_m$-Bergman metrics are invariant under biholomorphisms.
\end{prop}

\begin{proof}
Let $\Omega\in\mathcal{D}^{\rm b}$ be a bounded domain, and let $F$ be a biholomorphism on $\Omega$.
By the transformation formula for the classic Bergman kernels, we have
$$
K^{\rm B}_{\Omega,1}(z)=
K^{\rm B}_{F(\Omega),1}(F(z)){|\mathcal{J}(F(z))|}^2.
$$
This implies that
$$
\mu^{\rm B}_{F(\Omega),2}(F(z))
=
\frac{1}{K^{\rm B}_{F(\Omega),1}(F(z))}
=
\frac{{|\mathcal{J}(F(z))|}^{2}}{K^{\rm B}_{\Omega,1}(z)}
=
{|\mathcal{J}(F(z))|}^{2}
\mu^{\rm B}_{\Omega,2}(z).
$$
Inductively, if we have
$$
\mu^{\rm B}_{F(\Omega),m}(F(z))
=
{|\mathcal{J}(F(z))|}^{2(m-1)}
\mu^{\rm B}_{\Omega,m}(z),
$$
then the transformation formula in Proposition \ref{prop: transformation} implies that
$$
\mu^{\rm B}_{F(\Omega),m+1}(F(z))
=
\frac{1}{K^{\rm B}_{F(\Omega),m}(F(z))}
=
\frac{{|\mathcal{J}(F(z))|}^{2m}}{K^{\rm B}_{\Omega,m}(z)}
=
{|\mathcal{J}(F(z))|}^{2m}
\mu^{\rm B}_{\Omega,m+1}(z).
$$
Therefore, Theorem \ref{thm: invariance} imply that $\mathcal{M}^{\rm B}_m$-Bergman metrics are invariant under biholomorphisms.
\end{proof}

Recall that for a bounded pseudoconvex domain $\Omega$, we defined a sequence of admissible weight function
$$
\mu^{\rm KE}_{\Omega,m}
:=
e^{-(m-1)\varphi^{\rm KE}_{\Omega}}
=
\frac{1}{{\det(g^{\rm KE}_{\Omega})}^{(m-1)}}.
$$
Denotes the corresponding weighted Bergman space by
$$
\mathcal{A}^2_{{\rm KE}_{m}}(\Omega)
:=
\mathcal{A}^2_{\mu^{\rm KE}_{\Omega,m}}(\Omega)
=
\mathcal{A}^2(\Omega,e^{-(m-1)\varphi^{\rm KE}_{\Omega}}).
$$
Fix a point $p\in\Omega$ and a nonzero vector $X\in\mathbb{C}^n$.
Denotes the subsets of $\mathcal{A}^2_{{\rm B}_{m}}(\Omega)$ and $\mathcal{A}^2_{{\rm KE}_{m}}(\Omega)$ by
$$
\mathcal{E}^j_{{\rm B}_{m}}(\Omega)
:=
\mathcal{E}^j_{\mu^{\rm B}_{\Omega,m}}(\Omega),\ \ \ \ \ 
\mathcal{E}^j_{{\rm KE}_{m}}(\Omega)
:=
\mathcal{E}^j_{\mu^{\rm KE}_{\Omega,m}}(\Omega).
$$
Denotes the corresponding minimizers by $v^j_{{\rm B}_{m}}$ and $v^j_{{\rm KE}_{m}}$, i.e.,
$$
I^j_{{\rm B}_{m}}:=\inf_{\mathcal{E}^j_{{\rm B}_{m}}(\Omega)}\|u\|_{\mu^{\rm B}_{\Omega,m}}^2
=\|v^j_{{\rm B}_m}\|^2_{\mu^{\rm B}_{\Omega,m}},\ \ \ \ \
I^j_{{\rm KE}_{m}}:=\inf_{\mathcal{E}^j_{{\rm B}_{m}}(\Omega)}\|u\|_{\mu^{\rm KE}_{\Omega,m}}^2
=\|v^j_{{\rm KE}_m}\|^2_{\mu^{\rm KE}_{\Omega,m}}.
$$

\begin{lem}\label{ratio of minimum}
Let $\Omega$ be a bounded pseudoconvex domain in $\mathbb{C}^n$.
Suppose that
$
\frac{\mu^{\rm B}_{\Omega,m}}{\mu^{\rm KE}_{\Omega,m}}
$
is bounded from below and above.
Fix a point $p\in\Omega$ and a nonzero vector $X\in\mathbb{C}^n$.
Then, for $j=0,1,2$, we have
$$
\inf_{\Omega}\left(\frac{\mu^{\rm B}_{\Omega,m}}{\mu^{\rm KE}_{\Omega,m}}\right)
I^j_{{\rm KE}_{m}}
\leq
I^j_{{\rm B}_{m}}
\leq
\sup_{\Omega}\left(\frac{\mu^{\rm B}_{\Omega,m}}{\mu^{\rm KE}_{\Omega,m}}\right)
I^j_{{\rm KE}_{m}}.
$$
\end{lem}
\begin{proof}

By the boundedness assumption,
$$
\mathcal{E}^j_{{\rm B}_{m}}(\Omega)
=
\mathcal{E}^j_{{\rm KE}_{m}}(\Omega).
$$
The definition of the minimum integrals shows that
$$
I^j_{{\rm B}_{m}}
=
\|v^j_{{\rm B}_m}\|^2_{\mu^{\rm B}_{\Omega,m}}
\leq
\|v^j_{{\rm KE}_m}\|^2_{\mu^{\rm B}_{\Omega,m}}
\leq
\sup_{\Omega}\left(\frac{\mu^{\rm B}_{\Omega,m}}{\mu^{\rm KE}_{\Omega,m}}\right)
\|v^j_{{\rm KE}_m}\|^2_{\mu^{\rm KE}_{\Omega,m}}
=
\sup_{\Omega}\left(\frac{\mu^{\rm B}_{\Omega,m}}{\mu^{\rm KE}_{\Omega,m}}\right)
I^j_{{\rm KE}_{m}}.
$$
Similarly, we have
$$
I^j_{{\rm B}_{m}}
=
\|v^j_{{\rm B}_m}\|^2_{\mu^{\rm B}_{\Omega,m}}
\geq
\inf_{\Omega}\left(\frac{\mu^{\rm B}_{\Omega,m}}{\mu^{\rm KE}_{\Omega,m}}\right)
\|v^j_{{\rm B}_m}\|^2_{\mu^{\rm KE}_{\Omega,m}}
\geq
\inf_{\Omega}\left(\frac{\mu^{\rm B}_{\Omega,m}}{\mu^{\rm KE}_{\Omega,m}}\right)
\|v^j_{{\rm KE}_m}\|^2_{\mu^{\rm KE}_{\Omega,m}}
=
\inf_{\Omega}\left(\frac{\mu^{\rm B}_{\Omega,m}}{\mu^{\rm KE}_{\Omega,m}}\right)
I^j_{{\rm KE}_{m}}.
$$
\end{proof}

\begin{prop}\label{prop: ineq for minimum}
Set
$L_m:=\inf\limits_{\Omega}\left(\frac{\mu^{\rm B}_{\Omega,m}}{\mu^{\rm KE}_{\Omega,m}}\right)$ 
and
$U_m:=\sup\limits_{\Omega}\left(\frac{\mu^{\rm B}_{\Omega,m}}{\mu^{\rm KE}_{\Omega,m}}\right)$.
Then we have
$$
\frac{1}{U_m}
K^{\rm KE}_{\Omega,m}
\leq
K^{\rm B}_{\Omega,m}
\leq
\frac{1}{L_m}
K^{\rm KE}_{\Omega,m},\ \ \ \ \
\frac{L_m}{U_m}
g^{\rm KE}_{\Omega,m}
\leq
g^{\rm B}_{\Omega,m}
\leq
\frac{U_m}{L_m}
g^{\rm KE}_{\Omega,m},
$$
$$
\left(\frac{L_m}{U_m}\right)^2
\left(2-H^{\rm KE}_{\Omega,m}\right)
\leq
2-H^{\rm B}_{\Omega,m}
\leq
\left(\frac{U_m}{L_m}\right)^2
\left(2-H^{\rm KE}_{\Omega,m}\right).
$$
\end{prop}
\begin{proof}
Apply Lemma \ref{ratio of minimum} and Theorem \ref{thm: bergman-fuks}.
\end{proof}

\begin{rmk}
Note that if
$\lim\limits_{m\rightarrow\infty}\sqrt[m]{L_m}=\lim\limits_{m\rightarrow\infty}\sqrt[m]{U_m}=1$, 
the convergence of Tian's sequence of normalized Bergman kernels $\widetilde{K}^{\rm KE}_{\Omega,m}$ implies the convergence of $\sqrt[m]{K^{\rm B}_{\Omega,m}}$.
Moreover, if
$\lim\limits_{m\rightarrow\infty}\frac{U_m}{L_m}=1$,
then the convergence of Tian's sequence of normalized Bergman metrics $\widetilde{g}^{\rm KE}_{\Omega,m}$ implies the convergence of $\frac{1}{m}g^{\rm B}_{\Omega,m}$.
\end{rmk}

\begin{lem}\label{ratio estimate}
Let $\Omega$ be a uniform squeezing domain in $\mathbb{C}^n$ with the squeezing number $r\in(0,1]$.
For any integer $m\geq1$,
there exist uniform constants $C,D_m>0$ satisfying
$$
\frac{(\pi r^2)^{nm}}{C^m}
D_m
\leq
\frac{\mu^{\rm B}_{\Omega,m+1}}{\mu^{\rm KE}_{\Omega,m+1}}
\leq
C^m
\pi^{nm}
D_m.
$$
\end{lem}

\begin{proof}
For any point $p\in\Omega$,
$$
\frac{\mu^{\rm B}_{\Omega,m+1}}{\mu^{\rm KE}_{\Omega,m+1}}(p)
=
\frac{{\det(g^{\rm KE}_{\Omega})}^{m}}{K^{\rm B}_{\Omega,m}}(p)
=
\frac{{\det(g^{\rm KE}_{\Omega_p})}^{m}}{K^{\rm B}_{\Omega_p,m}}(0).
$$
Note that the uniform squeezing property implies that there exists a positive constant $C>0$ satisfying
$$
\frac{1}{C}\leq \det(g^{\rm KE}_{\Omega_p})(0) \leq C.
$$
Moreover, the squeezing property yields that
$$
K^{\rm B}_{\mathbb{B}^n(1),m}(0)
\leq
K^{\rm B}_{\Omega_p,m}(0)
\leq
K^{\rm B}_{\mathbb{B}^n(r),m}(0).
$$
The computation in Section \ref{Forelli-Rudin} shows that
\begin{align*}
K^{\rm B}_{\mathbb{B}^n(r),m}(0)
&=\prod^{m-1}_{k=0}\frac{1}{c_k(r)}
=\prod^{m-1}_{k=0}\left(\frac{1}{(\pi r^2)^n}\frac{(k(n+1)+n)!}{(k(n+1))!}\right)\\
&=\frac{1}{(\pi r^2)^{nm}}\frac{(mn+m-1)!}{(m-1)!(n+1)^{m-1}}
=:\frac{1}{(\pi r^2)^{nm}}\frac{1}{D_m}
.
\end{align*}
\end{proof}

\begin{rmk}
Unfortunately, $\lim\limits_{m\rightarrow\infty}D_m=0$ so that $\lim\limits_{m\rightarrow\infty}\sqrt[m]{L_m}=\lim\limits_{m\rightarrow\infty}\sqrt[m]{U_m}=0$.
However, since $\lim\limits_{m\rightarrow\infty}m^nD_m^{1/m}=\left(\frac{e}{n+1}\right)^n$,
a normalization of the convergence speed implies that
$$
\frac{r^{2n}}{C}\left(\frac{e}{n+1}\right)^n
\leq
\lim_{m\rightarrow\infty}
\sqrt[m]{\left(\frac{m}{\pi}\right)^{nm}\frac{\mu^{\rm B}_{\Omega,m+1}}{\mu^{\rm KE}_{\Omega,m+1}}}
\leq
C
\left(\frac{e}{n+1}\right)^n.
$$    
\end{rmk}

As in the case of compact canonically polarized manifolds (cf. \cite{tsuji2010,song2010}), for convergences, we consider a {\it modified} sequence of weight functions $\widetilde{\mu}^{\rm B}_{\Omega,m}$ for $m\geq1$ inductively as follows.
Set $\widetilde{\mu}^{\rm B}_{\Omega,1}:=1_{\Omega}$.
For the given admissible weight $\widetilde{\mu}^{\rm B}_{\Omega,m}$, consider the weighted Bergman space and the corresponding Bergman kernel:
$$
\mathcal{A}^2_{\widetilde{\mu}^{\rm B}_{\Omega,m}}(\Omega)
=
\mathcal{A}^2(\Omega,\widetilde{\mu}^{\rm B}_{\Omega,m}),\ \ \ \ 
K^{\rm \widetilde{B}}_{\Omega,m}
:=
K_{\Omega,\widetilde{\mu}^{\rm B}_{\Omega,m}}
.
$$
Define the next admissible weight function by
$$
\widetilde{\mu}^{\rm B}_{\Omega,m+1}
:=
\left(\frac{m}{\pi}\right)^n\left(1-\frac{n}{2m}\right)
\frac{1}{K^{\rm \widetilde{B}}_{\Omega,m}}
=
\left(\frac{m}{\pi}\right)^n\left(1-\frac{n}{2m}\right)
\frac{1}{K_{\Omega,\widetilde{\mu}^{\rm B}_{\Omega,m}}}.
$$
Following Berndtsson's idea in \cite{berndtsson2009},
we used the normalization factor $\left(\frac{m}{\pi}\right)^n\left(1-\frac{n}{2m}\right)$ rather than $\left(\frac{m}{\pi}\right)^n$ for a better convergence rate (cf. \cite{tsuji2010,song2010}).
The below theorem generalize Theorem 1.2 in \cite{tsuji2013dynamical}.

\begin{thm}\label{thm: tsuji seq converge}
Let $\Omega$ be a uniform squeezing domain in $\mathbb{C}^n$.
There exists a uniform constant $C>0$ satisfying
$$
e^{-\frac{C}{m}}
\leq
\sqrt[m]{\frac{\widetilde{\mu}^{\rm B}_{\Omega,m+1}}{\mu^{\rm KE}_{\Omega,m+1}}}
\leq
e^{\frac{C}{m}}.
$$
In particular, we have the following uniform convergence:
$$
\lim_{m\rightarrow\infty}\sqrt[m]{K^{\rm \widetilde{B}}_{\Omega,m}}
=
\det\left(g^{\rm KE}_{\Omega}\right).
$$
\end{thm}

\begin{proof}
Let $m_0$ be the constant in Theorem \ref{thm: asymptototic expansion} and Corollary \ref{cor: asymptototic expansion}.
Lemma \ref{ratio estimate} implies that there exist uniform constants $L_{m_0},U_{m_0}>0$ satisfying
$$
L_{m_0}
\leq
\frac{\widetilde{\mu}^{\rm B}_{\Omega,m_0}}{\mu^{\rm KE}_{\Omega,m_0}}
\leq
U_{m_0}.
$$
For the proof, we use mathematical induction on $m$.
Let $m\geq m_0$.
Assume that there exist uniform constants $\widetilde{L}_{m},\widetilde{U}_{m}>0$ satisfying
\begin{equation}\label{induction hypothesis}
\widetilde{L}_{m}
\leq
\frac{\widetilde{\mu}^{\rm B}_{\Omega,m}}{\mu^{\rm KE}_{\Omega,m}}
\leq
\widetilde{U}_{m}.
\end{equation}
From the above assumption, we want to show that there exists a constant $C>0$ satisfying
$$
\widetilde{L}_{m}
\left(1-\frac{C}{m^2}\right)
\leq
\frac{\widetilde{\mu}^{\rm B}_{\Omega,m+1}}{\mu^{\rm KE}_{\Omega,m+1}}
\leq
\widetilde{U}_{m}
\left(1+\frac{C}{m^2}\right).
$$
Apply the same proof of Lemma \ref{ratio of minimum} to $\frac{\widetilde{\mu}^{\rm B}_{\Omega,m}}{\mu^{\rm KE}_{\Omega,m}}$.
Then, the induction hypothesis (\ref{induction hypothesis}) imply that for any $p\in\Omega$,
\begin{equation}\label{ineq:proof1}
 \widetilde{L}_{m}
I^0_{{\rm KE}_{m}}(p)
\leq
I^0_{{\rm \widetilde{B}}_{m}}(p)
\leq
\widetilde{U}_{m}
I^0_{{\rm KE}_{m}}(p).   
\end{equation}
Recall that for all $m\geq m_0$, we have
\begin{align*}
I^0_{{\rm KE}_{m}}(p)
=
\frac{1}{K^{\rm KE}_{\Omega,m}(p)}
&=
\frac{1}{{\det(g^{\rm KE}_{\Omega}(p))}^{m}}
\left(\frac{\pi}{m}\right)^n
\left(1+\frac{n}{2m}+O\Big(\frac{1}{m^2}\Big)\right)\\
&=
\mu^{\rm KE}_{\Omega,m+1}(p)
\left(\frac{\pi}{m}\right)^n
\left(1+\frac{n}{2m}+O\Big(\frac{1}{m^2}\Big)\right).
\end{align*}
This implies that
$$
I^0_{{\rm KE}_{m}}(p)\left(\frac{m}{\pi}\right)^n\left(1-\frac{n}{2m}\right)
=
\mu^{\rm KE}_{\Omega,m+1}(p)\left(1+O\Big(\frac{1}{m^2}\Big)\right).
$$
In other words, there exists a constant $C>0$ satisfying
\begin{equation}\label{ineq:proof2}
\mu^{\rm KE}_{\Omega,m+1}(p)\left(1-\frac{C}{m^2}\right)
\leq
I^0_{{\rm KE}_{m}}(p)\left(\frac{m}{\pi}\right)^n\left(1-\frac{n}{2m}\right)
\leq
\mu^{\rm KE}_{\Omega,m+1}(p)\left(1+\frac{C}{m^2}\right).   
\end{equation}
On the other hand, by the definition, we have
$$
I^0_{{\rm \widetilde{B}}_{m}}(p)
=
\frac{1}{K^{\rm \widetilde{B}}_{\Omega,m}(p)}
=
\frac{1}{\left(\frac{m}{\pi}\right)^n\left(1-\frac{n}{2m}\right)}
\widetilde{\mu}^{\rm B}_{\Omega,m+1}(p).
$$
From the inequalities (\ref{ineq:proof1}) and (\ref{ineq:proof2}), we obtain that
$$
\widetilde{L}_{m}
\left(1-\frac{C}{m^2}\right)
\leq
\frac{\widetilde{\mu}^{\rm B}_{\Omega,m+1}}{\mu^{\rm KE}_{\Omega,m+1}}
\leq
\widetilde{U}_{m}
\left(1+\frac{C}{m^2}\right)
$$
as we required, which completes the induction step.
Therefore, it follows that
$$
L_{m_0}
\prod_{k=m_0}^m
\left(1-\frac{C}{k^2}\right)
\leq
\frac{\widetilde{\mu}^{\rm B}_{\Omega,m+1}}{\mu^{\rm KE}_{\Omega,m+1}}
\leq
U_{m_0}
\prod_{k=m_0}^m
\left(1+\frac{C}{k^2}\right)
.
$$
Finally, the conclusion follows from the following upper bound estimates:
$$
\log\left(\prod_{k=m_0}^m
\left(1+\frac{C}{k^2}\right)\right)
\leq
\sum_{k=1}^m\left(\log\left(1+\frac{C}{k^2}\right)\right)
\leq
C'\sum_{k=1}^m\frac{1}{k^2}
\leq
C''.
$$
Changing the signs of the above inequalities yields the lower bound estimates.
\end{proof}

\textbf{Acknowledgements}. 
The author would like to thank Professor Kang-Tae Kim for his suggestion of this work and valuable comments.
He also would like to thank Professor Jun-Muk Hwang for his encouragement to write this paper, and Professor Bo Berndtsson for sharing the note \cite{berndtsson2009}.
This work was supported by Incheon National University Research Grant in 2022.
\bibliographystyle{abbrev}
\bibliography{reference}

\end{document}